\documentclass[11pt]{article}
\usepackage[margin=25mm]{geometry}
\usepackage{amsmath}
\usepackage{amsthm}
\usepackage{amsfonts}
\usepackage{amssymb}
\usepackage{graphicx}
\usepackage{verbatim}
\usepackage{empheq}
\usepackage{newfloat}
\immediate\write18{texcount -tex -sum  \jobname.tex > \jobname.wordcount.tex}

\newtheorem{thm}{Theorem}[section]

\newtheorem{prop}[thm]{Proposition}
\newtheorem{rem}[thm]{Remark}

\usepackage{algorithm,algpseudocode}
\algnewcommand{\LeftComment}[1]{\Statex \(\triangleright\) #1}

\algrenewcommand\algorithmicrequire{\textbf{Input:}}
\algrenewcommand\algorithmicensure{\textbf{Output:}}

\usepackage{subfigure}
\usepackage{float}

\graphicspath{{imagesIM/}}

\usepackage{verbatim}
\usepackage{ifthen}

\usepackage{caption}

\usepackage{mathtools}

\usepackage{comment}

\usepackage{hyperref}

\hypersetup{
    colorlinks=true,
    linkcolor=blue,
    filecolor=magenta,      
    urlcolor=cyan,
}

\DeclareFloatingEnvironment[
  fileext=lob,
  listname={List of Boxes},
  name=Box,
  placement=htp,
]{BOX}

\DeclareMathOperator{\Tr}{Tr}

\DeclareMathOperator*{\argmin}{arg\,min}

\newcommand{\red}[1]{{}}

\newcommand{\blue}[1]{#1}

\providecommand{\keywords}[1]
{
  \small	
  \textbf{\text{Keywords: }} #1
}

\title{}
\title{\textbf{Adjoint-based optimal control of contractile elastic bodies. Application to limbless locomotion on frictional substrates}}
\author{Ashutosh Bijalwan$^{\blue{1,2}}$, Jos\'e J. Mu\~noz$^{1,2,3,4}$  \\
       \small $^{1}$Universitat Polit\`ecnica de Catalunya (UPC), Barcelona, Spain \\
        \small $^{2}$Centre Internacional de Mètodes Num\`erics en Enginyeria (CIMNE), Barcelona, Spain \\
        \small $^{3}$Dept. of Mathematics, UPC, Barcelona, Spain \\
        \small $^{4}$ Institut de Matem\`atiques de la UPC - BarcelonaTech, Barcelona, Spain. \\
}

\date{} 

\newcommand{\vect}[1]{\boldsymbol{#1}}
\newcommand{\qvect}[1]{\mathbf{#1}}

\newcommand{\xf}{\vect x}
\newcommand{\x}{\qvect x}
\newcommand{\bvf}{\vect v}
\newcommand{\bv}{\qvect v}
\newcommand{\buf}{\vect u}
\newcommand{\bu}{\qvect u}
\newcommand{\bpf}{\vect p}
\newcommand{\bqf}{\vect q}
\newcommand{\bp}{\qvect p}
\newcommand{\bq}{\qvect q}

\newcommand{\0}{\vect 0}

\newcommand{\zi}{\vect i}
\newcommand{\zx}{\vect X}
\newcommand{\N}{\vect N}
\newcommand{\n}{\vect n}
\newcommand{\zt}{\vect t}
\newcommand{\br}{\vect r}
\newcommand{\bs}{\vect s}
\newcommand{\e}{\vect e}
\newcommand{\E}{\vect E}
\newcommand{\g}{\vect g}
\newcommand{\bc}{\vect c}
\newcommand{\za}{\vect a}
\newcommand{\zb}{\vect b}
\newcommand{\dd}{\vect d}

\newcommand{\y}{\vect y}
\newcommand{\z}{\vect z}

\newcommand{\veps}{\varepsilon}
\newcommand{\bl}{\vect \lambda}
\newcommand{\et}{\vect \eta}
\newcommand{\tu}{\vect \tau}

\newcommand{\bh}{\vect \zeta}
\newcommand{\bx}{\vect \xi}
\newcommand{\bm}{\vect \mu}

\newcommand{\B}{\mathbf{B}}

\newcommand{\M}{\mathbf{M}}
\newcommand{\K}{\mathbf{K}}
\newcommand{\G}{\mathbf{G}}
\newcommand{\zP}{\mathbf{P}} 
\newcommand{\R}{\mathbf{R}}
\newcommand{\zF}{\mathbf{F}}
\newcommand{\zC}{\mathbf{C}}
\newcommand{\zE}{\mathbf{E}}
\newcommand{\zL}{\mathbf{L}}

\begin{document}
\maketitle

\begin{abstract}

In nature, limbless locomotion is adopted by a wide range of organisms at various length scales. Interestingly, undulatory, crawling and inching/looping gait constitutes a fundamental class of limbless locomotion and is often observed in many species such as caterpillars, earthworms, leeches, larvae, and \blue{\emph{C. elegans}}, to name a few. In this work, we developed a computationally efficient 3D Finite Element (FE) based unified framework for the locomotion of limbless organisms on soft substrates. \blue{Muscle} activity is simulated with a multiplicative decomposition of deformation gradient, which allows mimicking a broad range of locomotion patterns in 3D solids on \blue{frictional}  substrates. In particular, a two-field FE formulation \blue{based on positions and velocities} is proposed. \blue{Governing} partial differential equations are transformed into equivalent time-continuous differential-algebraic equations (DAEs). Next, the optimal locomotion strategies are studied in the framework of optimal control theory. \blue{We resort to} adjoint-based methods and deduce the first-order optimality conditions, that yield a system of DAEs with two-point end conditions. Hidden symplectic structure and Symplectic Euler time integration of optimality conditions have been discussed. The resulting discrete first-order optimality conditions form a \blue{non-linear programming} problem that  is solved efficiently with the Forward Backwards Sweep Method. Finally, some numerical examples are provided to demonstrate the comprehensiveness of the proposed computational framework and investigate the energy-efficient optimal limbless locomotion strategy out of distinct locomotion patterns adopted by limbless organisms.
\end{abstract} \hspace{10pt}


\section{Introduction}\label{s:Intro}

\blue{In this paper we present a formulation for computing the optimal evolution of contractility profiles in elastic bodies that are in contact with rigid substrates, subjected to a velocity dependent frictional law. The problem is numerically solved and applied to the locomotion of slender bodies. In this section we review the approaches, comment our contributions, and explain our choices in relation to: i) soft body locomotion, ii) optimal control, and iii) its numerical solution.}

Locomotion \blue{of limbless deformable body organisms} and their adaptability to changing environments has inspired the design and development of complex soft robots, bio-medical research and many more \cite{shepherdm11}. \blue{For instance,} soft body robots can be employed the surveillance of human colon ulcer \cite{wuwireless22} and targeted drug delivery on desired sites \cite{xumicro22}. Limbless organisms perform locomotion in their lifetime to fulfil vital functions such as food search, reproduction \blue{or as a defence strategy}. The energy efficiency of specific locomotion patterns and underlying control mechanism is not fully understood and is still an open question \blue{for} the scientific community. To address these questions, researchers often \blue{resort to} principles of optimal control theory \blue{\cite{sharp21,lenhart07}. However}, soft-body limbless locomotion exhibits \blue{large deformations}, negligible inertia and interface anisotropy\blue{, which make} the depiction of the control laws notoriously hard.  

Limbless organisms propel in a medium by the periodic deformation of their body shape \blue{ through coordinated and  systematic  contraction and extension of their} ventral-dorsal skeletal muscle system. Although \blue{these contractions are necessary}  for locomotion, environmental frictional conditions are equally important. In flat and isotropic frictional substrates, limbless locomotion results in no net displacement of body centroid, and interface anisotropy emerged as necessary conditions for locomotion \cite{gray64, munoz22}. Nature provides \blue{efficient strategies} to break frictional isotropy and furnish out undulatory, crawling and inching gait patterns. Typical illustrations \blue{of these patterns are \emph{Caenorhabditis elegans} (\emph{C. elegans})}, larvae, and caterpillars, respectively. For instant, \blue{\emph{C. elegans}} breaks the frictional isotropy over agar substrate through a thin film hydrodynamic lubrication and introduce higher drag forces along body normal \cite{rabets14,shen12}. Instead larvae have a segmented body structure while stiff bristles/setae are present on the earthworm body which helps in inducing frictional anisotropy along the forward and reverse direction of advancement \cite{fang15compr}. On the contrary, caterpillar breaks frictional isotropy with the popular anchoring mechanism \cite{xu22loco}. These worms have a unique ability to modulate muscular wave characteristics with the sensitivity of substrate friction. For instant \blue{\emph{C. elegans}} in a small frictional environment exhibits a swimming gait with a low frequency of undulation whereas, in high frictional interfaces, they show crawling gain with a high frequency of undulation \cite{karrakchouOP06, fangb10}. The identification of the regulatory mechanism behind this adaptability is not completely understood. It becomes thus interesting to study the optimal trajectories and compare them with the gait adopted by the limbless organism.

Motivated by the seminal work of Gray \cite{gray64} and Purcell \cite{purcell77}, the \blue{analysis of locomotion modes} has significantly improved the development of soft robots and its success paved new avenues for the computational design. The challenging aspect in the modelling front is the finite deformation behaviour of soft continuum robots under service conditions to ensure the intended functionality. The last two decades witnessed extensive work towards the development of beam and Cosserat rod-based planar and spatial models to address these computational issues \cite{alouges13,bilbao18roll, ranner20, munoz22, bijjm23}. The accomplishment of these multibody system modelling techniques lies in the fact that fewer degrees of freedom are sufficient to predict the global response \cite{hemingway21, wang23mono}. However, the rod-based model represents the soft body cross-sectional geometry in an average sense and often misses out  the underlying bio-mechanics responsible for the bulk response. Additionally, these approaches include inertial effects to model the locomotion of limbless organisms such as larvae, inchworms, and earthworms. 
Unfortunately, the presence of small inertia often introduces material waves \blue{that} interfere with the muscular travelling waves and produce unacceptable non-physical state trajectories. On the contrary, neglecting inertial effects leads to a non-canonical form of the equilibrium equations which {demands special care for not losing} the underlying geometric structure of the solution.

\blue{In this work, while neglecting inertial effects, we preserve the symplectic structure of the equations, even in the presence of the mentioned velocity dependent non-isotropic frictional conditions. The problem is formulated in the context of finite elasticity, where time dependent growth is included by decomposing the deformation gradient into an elastic and a growth component \cite{rodriguez94}. We propose a two-field formulation based on positions and velocities, and resort to Finite Element (FE) space discretisation of the momentum balance Partial Differential Equations (PDE), which result in a set of index-1 time-continuous non-linear Differential-Algebraic Equations (DAE) \cite{ascher98}. The latter are included as a constraint to our optimisation problem, which minimises a functional that measures the distance of the body centre of mass with respect to a target position.}

\blue{The application of optimal control problems (OCP) arises for instance} in trajectory planning, cancer research, cardiovascular modelling, or epidemiology \cite{knopoff13, millerh15, karrakchou06}. \blue{The numerical solution of  OCP admits two popular approaches: first-discretise then-optimise (DO) or first-optimise then-discretise (OD) \cite{betts10, bryson75}.} The optimality conditions of the latter approach constitute Euler-Lagrange  \blue{equations} with a two-point boundary condition and the so-called Hamiltonian Boundary Value Problem (HBVP). \blue{Its time discretisation turns the HBVP into} the so-called non-linear programming problem (NLP) \cite{betsch17, bijjm23}.  \blue{Problems subjected to PDE constraints are often more difficult to solve} \cite{hinze09, troltzsch10}. We opt to a space discretisation of the PDE stemming from our elasticity problem with velocitiy dependent friction. 

So far there are quite a few articles addressing the OD approach in the finite elasticity. Optimal control of soft tissue have been thoroughly discussed in \cite{lubkoll14, gunnel16, martinez20} \blue{for elliptical Boundary Value Problem (BVP) and associated discrete OCP in the context of 3D finite elasticity}. In particular, \cite{martinez20} shows the potential application of the OD approach on the computational design of soft actuators, where optimal shapes are obtained with the gradient descent strategy\blue{, but in elastostatics}. For hyperbolic PDEs, the space and time discretisation leads to a large size NLP and recasting the optimality conditions into the standard control Hamiltonian-based framework with OD approach is often difficult due to the natural boundary conditions \cite{troltzsch10}. Similarly, the DO approach results in a NLP with the annihilation of associated geometric structures which could be utilised in designing structure-preserving integrators \cite{flasskamp19, betsch17, bijjm23}.

Various time-discretisation schemes have been explored for the numerically solving HBVP \blue{that preserve the underlying geometric structure} \cite{hairer02}. Classical Implicit Euler, Mid-Point  and Energy-Momentum schemes do not preserve \blue{symplecticity} of the solution \cite{flasskamp19, betsch17, bijjm23}. For practical purposes, the numerical stability of discretised OCP is sensitive to the regularisation of control effort and these discretisation strategies are not immune to this. Recently, it has been shown by the authors \blue{that the symplectic Euler (SE) scheme applied to OCP} shows a stable optimal solution trajectory \cite{ABJJM23}. \blue{In this work, we present a general setting for a wider set of integration schemes, which we will particularise to the SE scheme}. 

In practice, the time-step size is often kept small, which results in a large-scale \blue{NLP suffering} from the curse of dimensionality. Employment of the Hessian-based method becomes computationally expensive\blue{, and } deducing the analytic Hessian of optimality conditions is a tedious task, error-prone and usually unavailable. \blue{Alternatively}, populating a numerical Hessian with a finite-difference approximation \blue{destroys in general its sparsity and symmetry, and may lead to an ill-conditioned system. For this reason, we resort to Forward Backward Sweep Method (FBSM) \cite{sharp21, lenhart07}, where the use of Hessian matrix is avoided. In this method, the state and adjoint DAEs are solved separately, and the control variable is updated iteratively.} 

The presence of inequality constraints \blue{for the control variable} introduces discontinuities and additional numerical challenges. \blue{Primal-dual} interior point methods form equivalent smooth problems by introducing a barrier parameter. Unfortunately, the smoothed problem may become ill-conditioned for the small value of the barrier parameter and also suffers from the curse of dimensionality \cite{nocedal06}. Our implementation of \blue{the FBSM for HBVP in conjunction with a gradient-based method also incorporates a projection step that imposes the required inequalities in a satisfactory manner. Some popular choices of } gradient-based methods are Generalised Minimal Residual, Conjugate Gradient (CG), and Gradient Descent (GD). Among them, CG and GD can be extended to  non-linear systems \cite{lasdon67}. A major drawback of the CG method is that it requires the system to be symmetric definite, \blue{a propertie that OCP do not necessarily satisfy, depending on the time-discretisation employed}. \blue{Although GD methods do not suffer from this restriction, they } often take a large number of iterations to locate the extrema. \blue{In our implementation, we accelerate the convergence of GD} with a \blue{specific line-search} scheme based on the Barzikai-Borwein algorithm \cite{barzilai88,burdakov19,sharp21}.


After a comprehensive glimpse at the available literature, it can be concluded that beam mechanics-based planer and spatial models have been explored by researchers to model the specific locomotion of the limbless organism \blue{\cite{hemingway21,alouges13,munoz22}, but not for computing optimal motion. While} optimal control of soft bodies undergoing finite deformations is extensively studied for elliptical PDE constraint only (quasi-static scenario), to the best of our knowledge, a unified finite element-based computational framework for limbless locomotion taking into account muscle active response, muscles orientation, and substrate anisotropy is not available in the literature. In this work, we present a first attempt to develop a consistent computational framework for the forward and associated time-dependent optimal control problem. \blue{In our formulation,  upon introducing a FE discretisation, we obtain} a  DAE constrained optimisation problem that is valid for general three dimensional contractile bodies, with muscle activity modelled through growth patterns that mimic muscle contraction and extension.  \blue{Our approach has two main} advantages: i) it constitutes a general framework for a larger class of limbless locomotion, and ii) the presented framework automatically furnishes the desired time-continuous optimality conditions with the inherent symplectic structure of the solution.

The paper is organised as follows. In Section \ref{s:compu} we describe the governing laws associated with the limbless locomotion on the soft substrates. Here with two-field formulation, the strong form of the PDEs is reduced to the equivalent DAEs. The optimal control problem and optimality conditions are introduced in Section \ref{s:OCP}. Numerical time integration of the Euler-Lagrange equations is described in Section \ref{s:timedisc} followed by the implementation of the optimisation algorithm in Section \ref{s:oalgo}. Central ideas have been validated in Section \ref {s:numexp} by performing numerical experiments on the forward uncontrolled and controlled limbless locomotion on the soft substrates. Finally, the conclusions and future perspectives are drawn in Section \ref{concl}. Information about the sensitivities,  tangent matrices, and \blue{line-search process are given in Appendix \ref{appendix:a} and \ref{appendix:b}.}

\textbf{Notation}: $\mathbb{V}$ is a space of vectors, $\mathit{Lin}$ is space of linear mapping from $\mathbb{V} \times \mathbb{V}$, $\mathit{Sym}$ is a space of symmetric second order tensor, and $\mathbb{SO}(3)$ is the special orthogonal group. $\mathbb{R}_0^{+}$ is a space of all positive real numbers including zero. Suppose, $\za, \zb, \bc, \zx \in \mathbb{V}$, and $\zF, \R, \G \in \mathit{Lin}$, and $\mathbb{B} \in \mathit{Lin} \times \mathit{Lin}$. Then, $\dot{\za}=\frac{d \za}{dt} \bigr\rvert_{\zx} $ is material time-derivative and $\nabla_{\zx}:=\frac{\partial}{\partial \zx}$ is gradient operator defined in the reference configuration. The tensor product between two vector $\za$, $\zb$ is a second order tensor $\G=\za \otimes \zb$ with $G_{ij}=a_i b_j$. \blue{Property} $(\za \otimes \zb)\bc=(\zb \cdot \bc) \za$ is extensively used \blue{throughout} the article. The scalar product between two second-order tensors is defined as $\G:\R=G_{ij} R_{ij}$ with the usual Einstein's summation rule over repeated indices. The tensor product between two second order tensor $\zF$, $\R$ is a fourth-order tensor $\mathbb{B}=\zF \otimes \R$ with $\mathbb{B}_{ijkl}=F_{ij} R_{kl}$. Similarly, $\mathbb{B}=\zF \overline{\otimes} \R$ with $\mathbb{B}_{ijkl}=F_{il} R_{jk}$ and $\mathbb{B}=\zF \underline{\otimes} \R$ with $\mathbb{B}_{ijkl}=F_{ik} R_{jl}$. For instant, fourth rank identity tensor ($\mathbb{I}:=\mathbf{I} \underline{\otimes} \mathbf{I}$) is defined as $\mathbb{I}_{ijkl}=\delta_{ik}\delta_{jl}$, with $\mathbf{I}=\delta_{ij} \e_i \otimes \e_j$ as second rank identity tensor. Space of square integrable functions is defined as $L^2(\Omega_o):=\left\{ x_i \mid \int_{\Omega_o} x_i^2 d\Omega_o < \inf \right\}$ while $H^1(\Omega_o):= \left\{x_i \in L^2(\Omega_o) \mid \frac{\partial x_i}{\partial X_j} \in L^2(\Omega_o) \right\}$ is standard Hilbert or Sobolev space of degree one. $H_o^1(\Omega_o):= \left\lbrace p_i \in H^1(\Omega_o) \mid p_i=0 \text{ on } \Gamma_o^x \right\rbrace $ where $\Gamma_o^x$ is boundary of $\Omega_o$ with prescribed primary field. $W^{1, \infty}$ is a Sobolev space of continuous functions and essentially a Banach space and corresponding product space $W^{1,\infty}(\mathcal{I},\mathbb{R}^n):=W^{1, \infty}(\mathcal{I}) \times W^{1, \infty}(\mathcal{I}) \times \dots \times W^{1, \infty}(\mathcal{I})$ with vector mapping from time-domain $\mathcal{I}$ to $\mathbb{R}^n$.

\section{Computational framework} \label{s:compu}

In this section, we briefly summarise the continuum growth kinematics, governing laws for bulk and interface interactions, the strong form of the conservation laws and the corresponding weak solution procedure. In particular, \blue{we present} a unified two-field formulation \blue{based on positions and velocities} for the motion of solids in the absence of inertial effects.

\subsection{Kinematics of growth}

Consider an undeformed solid at $t=0$ placed in the reference configuration $\Omega_o$, and each (Lagrangian) material point  labelled with a unique position vector $\zx=X_i \E_i $ w.r.t to a fixed coordinate system equipped with a fixed, right-handed, orthonormal bases vector $\E_i \in \mathbb{R}^3$. We assume this configuration is an open, bounded, and connected subset of 3D Euclidean space ($\Omega_o \subseteq \mathbb{R}^3$) with a smooth boundary $\Gamma_o$ with normal $\hat{\N}(\zx)$. At any time $t>0, t \in \mathcal{I} \subset \mathbb{R}_0^{+}$, we assume the body occupies a unique deformed/current configuration $\Omega_t \subseteq \mathbb{R}^3$ with smooth boundary $\Gamma_t$ and material point $\zx$ is mapped to the corresponding spatial position $\xf :=\vect{\varphi}(\zx,t)$\blue{. The deformation tensor is given by $\zF:=\nabla_{\zx}\vect\varphi$, provided the map $\vect\varphi$} is injective and orientation preserving\blue{, i.e. $J:=\text{det}\zF > 0, \forall t\ge 0$} (see Fig. \ref{fig:1}a) \cite{bonet97, gurtin82}. Velocity of the material point with label $\zx$ is the material time derivative of the spatial position $\xf$ i.e., $\bvf(\zx,t):=\frac{\partial \xf(\zx,t)}{\partial t} \bigr\rvert_{\zx} $. \blue{In subsequent derivation, we may remove the dependence of scalars, vectors and tensors on $\vect X$ or $t$ for clarity.}

At each time instant $t \in \mathcal{I}$, the motion is decomposed into stress-free growth (often incompatible) to the virtual intermediate configuration $\overline{\Omega}_t$, followed by a pure elastic deformation which maps $\overline{\Omega}_t$ to the current configuration $\Omega_t$. \blue{The determinants} $J_e:=\det(\zF_e)$ and $J_g:=\det(\zF_g)$ are reversible elastic volume change and irreversible volume change induced by growth, respectively. We assume the undeformed solid undergoes \blue{anisotropic growth} deformation along the preferred direction $\zi(\zx)$ predefined in the reference configuration $\Omega_o$. Furthermore, we introduce a growth-related scalar field/internal variable $\blue{u(\vect X, t)} \in \mathbb{R}$ that physically represents the growth-induced stretching and plays a key role in the system evolution \cite{taber04}. The growth deformation gradient \blue{$\zF_g(\vect X,t)$} and the multiplicative decomposition of the deformation gradient \blue{$\zF(\vect X,t)$} can be expressed as
\begin{align*}
\zF:=& \zF_e \zF_g, \\
\zF_g:=& \mathbf{I}+ u \mathbf{A}, \\
J:= & J_e Jg,
\end{align*}
where $\mathbf{A}:=\zi \otimes \zi \in \mathit{Sym}$ is structural tensor corresponding to the fiber direction $\zi$ in the reference configuration, and \blue{$u$} determines the amount of growth or contraction, such that $\det(\mathbf F_g)=1+u$.

\blue{In subsequent derivations we will resort to the elastic part of the Green-Lagrangian strain tensor ($\zE:=\frac{1}{2}(\zC-\mathbf{I})$), written in terms of the invariant Right Cauchy-Green tensor $\zC:=\zF^\mathsf{T} \zF $, and the growth velocity gradient $\zL_g$, defined by} \cite{rodriguez94}
\begin{align*}
\zL_g:= & \dot{\zF}_g \zF_g^{-1}, \\
\zE_e:= & \frac{1}{2}(\zF_e^\mathsf{T} \zF_e-\mathbf{I}).
\end{align*}

\subsection{Helmholtz fee energy and thermodynamic restrictions}

\blue{Let us introduce} $\psi^e$ and $\psi$ as Helmholtz free energy of the solid per unit volume in intermediate ($\overline{\Omega}$) and reference configuration ($\Omega_o$), respectively. Then, total internal energy ($\mathcal{U}$) associated with the elastic deformation is given by \cite{lamm22}
\begin{align}\label{e:a3}
\begin{aligned}
\mathcal{U} &:=\int \limits_{\overline{\Omega}} \psi^e d\overline{\Omega}=\int \limits_{\Omega_o} \psi d\Omega_o, \\
\psi &:=J_g \psi^e.
\end{aligned}
\end{align}

Growth phenomenon must follow Clausius-Duhem inequality and under isothermal process takes the form \cite{gurtin82}
\begin{align}\label{e:a4}
\mathcal{D}_{int}:=-\dot{\psi}+\zP : \dot{\zF} \geq 0.
\end{align}

For simple elastic solids, the time derivative of the Helmholtz free energy per unit volume in reference configuration is expressed as
\begin{align}\label{e:a5}
\dot{\psi}=J_g \dot{\psi}^e+J_g \psi^e \zF_g^{-T} : \dot{\zF}_g.
\end{align}

Substituting Eq. \eqref{e:a5} in Eq. \eqref{e:a4}, internal dissipation inequality reduces to
\begin{align}\label{e:a6}
\mathcal{D}_{int}:=\left(\zP -J_g \frac{\partial \psi^e}{\partial \zF_e} \zF_g^{-T} \right) : \dot{\zF}+ \left( \zF_g^{T}\frac{\partial \psi^e}{\partial \zF_e}-\psi^e \mathbf{I}\right) : \zL_g \geq 0
\end{align}

As per the standard Noll-Coleman procedure, above inequality must be satisfied for all admissible elastic processes\blue{. Therefore,} the following restrictions on the growth process are deduced,
\begin{align*}
\zP = & J_g \zP_e \zF_g^\mathsf{-T}, \\
\mathcal{D}_{int}= & \R : \zL_g \geq 0,
\end{align*}
where $\zP_e:=\frac{\partial \psi^e}{\partial \zF_e}$ is \blue{the} elastic First Piola-Kirchhoff stress tensor, and tensor $\R:=\zF_g^\mathsf{T}\zP_e-\psi^e \mathbf{I}$ is the driving force for the growth phenomenon (muscles active response).

\begin{rem} \emph{Solids for which internal dissipation vanishes ($\mathcal{D}_{int}=0$) for all admissible deformation processes, and for which stress measure can be derived from a conservative potential are known as hyperelastic or Green elastic solids. In these cases the deformation phenomenon of an elastic solid can be modelled with a hyperelastic constitutive law.}
\end{rem}

In this work, the Neo-Hookean hyperelastic model is chosen and free energy is expressed in terms of the \blue{elastic Green-Lagrangian strain tensor} $\zE_e$ and Lame's \blue{parameters viz. shear and bulk modulus $\mu\in \mathbb{R}^+$ and $\lambda\in \mathbb{R}^+$, respectively}. Then, Helmholtz free energy of the Neo-Hookean solid per unit volume in intermediate configuration is introduced as \cite{wriggers08, bonet97}
\begin{align}\label{e:a9}
\psi^e :=\frac{\lambda}{2} (\ln{J_e})^2-\mu \ln{J_e}+\mu \Tr{(\zE_e)},
\end{align}
where $\Tr{(\zE_e)}=\frac{1}{2} (\blue{\zF_e:\zF_e} -3)$ is first invariant of $\zE_e$.
Using Eq. \eqref{e:a9}, elastic first Piola-Kirchhoff stress tensor $\zP_e$ and referential elastic fourth-order constitutive tensor $\mathbb{A}_e$ are given by
\begin{align}\label{e:a10}
\zP_e &:=\frac{\partial \psi^e}{\partial \zF_e}=\mu \zF_e+(\lambda \ln{J_e}-\mu) \zF_e^\mathsf{-T}, \\
\mathbb{A}_e &:= \frac{\partial \zP_e}{\partial \zF_e} = \mu \mathbb{I} -(\lambda \ln{J_e}-\mu) (\zF_e^\mathsf{-T} \overline{\otimes}\zF_e^{-1} )+\lambda \zF_e^\mathsf{-T} \otimes \zF_e^\mathsf{-T}.
\end{align}
In component form, referential elastic fourth-order constitutive tensor reduces to
\begin{align}\label{e:a11}
(\mathbb{A}_e)_{ijkl} := \frac{\partial (\zP_e)_{ij}}{\partial (\zF_e)_{kl}} = \mu \delta_{ik} \delta_{jl}-(\lambda \ln{J_e}-\mu) (\zF_e^{-1})_{li} (\zF_e^{-1})_{jk} +\lambda (\zF_e^{-1})_{ji} (\zF_e^{-1})_{lk} .
\end{align}

\begin{figure}[!htb]
    \centering
   \includegraphics[width=0.5\textwidth]{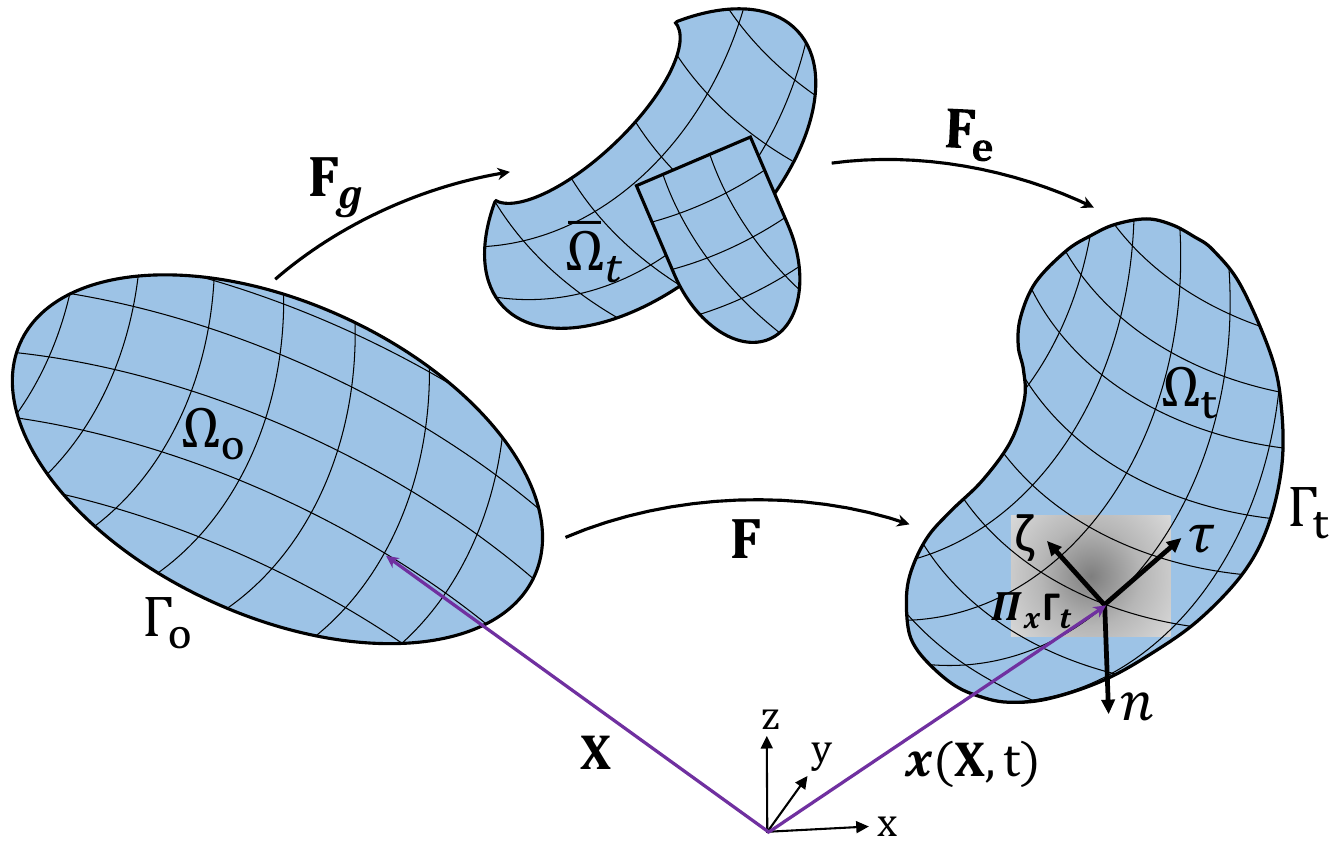} \hspace{3ex}
   \includegraphics[width=0.45\textwidth]{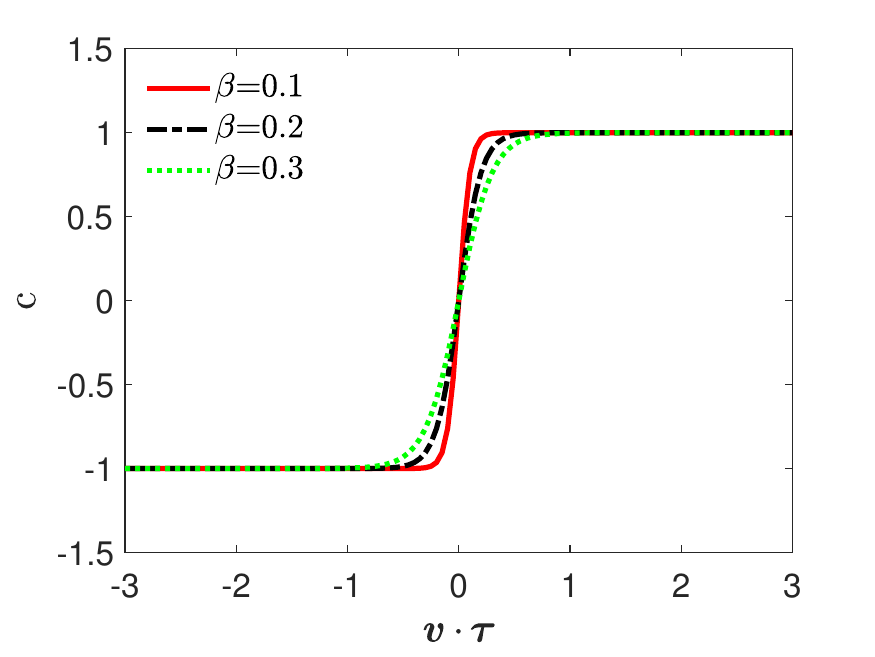}
    \caption{(a) Continuum 3D growth model, and (b) smoothed hyperbolic-tangent curve.}
   \label{fig:1}
\end{figure}

\begin{prop}\label{p:HP}
\emph{For prescribed finite growth, $u \in L^2(\Omega_o)$, total free energy $\psi$ is polyconvex and coercive for the compressible Neo-Hookean solid \blue{as defined in Eq. \eqref{e:a9}}.}
\end{prop}
\begin{proof}
Suppose at time $t \in \mathcal{I}$, growth field is given and bounded i.e., $u \in L^2(\Omega_o)$ with $J_g=1+u$ (see Appendix \ref{appendix:a4}). Then, total free energy can be solely expressed as a function of elastic invariants
\begin{align*}
\psi(u,\zF_e):=J_g(u) \psi^e(\zF_e) = \psi(\zF_e)
\end{align*}

With the above reduction and following lemma 6.5 of \cite{neff03}, polyconvexity of $\psi(\zF_e)$ can be established. The coercivity of $\psi(\zF_e)$ follows from the coercivity of elastic free energy of compressible Neo-Hookean solid \blue{that has a free energy given by $\psi^e(\zF_e)$ in Eq. \eqref{e:a9}}\cite{martinez20, schroder03} and requires $J_g(u) > 0$\blue{. Then}
\begin{align*}
\psi^e(\zF_e)& \geq c_1 \left(\zF_e:\zF_e+J_e^2 \right) + c_2, \\
\psi(\zF_e)& \geq c_1 J_g(u) \left(\zF_e:\zF_e+J_e^2 \right) + c_2 J_g(u),
\end{align*}
where $c_1 \in \mathbb{R}^+$ and $c_2 \in \mathbb{R}_0^+$.

$J_g(u) \in \mathbb{R}^+ $ represents the physical condition and rules out self-penetration under muscle active contraction. For a given bounded growth field, a polyconvex and coercive free energy function is a sufficient condition to ensure the existence of at least one minimiser of the energy functional $\mathcal{G}$ (see Section \ref{s:compu2field}) \cite{ball76}.

\end{proof}


\subsection{Substrate traction field}\label{s:compusub}

\blue{Let us denote by $\mathit{\Pi} \Gamma_t$ denotes the tangent bundle of the surface $\Gamma_t$ at time $t \in \mathcal{I}$}.
At each $\xf(\zx,t) \in \Gamma_t$, we introduce a set of orthonormal vectors $\{ \n, \tu, \bh \}$ defining the contact geometry. Vector $\n$ defines the substrate normal and the tangent space at $\xf$, $\mathit{\Pi}_{\xf}\Gamma_t$ is spanned by the vector $\tu$ (tangential) and $\bh$ (lateral) (See Fig. \ref{fig: 1t}). At $\xf(\zx,t)$, $\bvf_n$ and $\bvf_{\pi}$ are the projected velocity along the surface normal $\n$ and tangent plane $\mathit{\Pi}_{\xf}\Gamma_t$, respectively. We introduce a projection tensor field $\mathbf{\Upsilon}_n:=\mathbf{I}-\n \otimes \n$ such that $\mathbf{\Upsilon}_{\n} \bvf \in \mathit{\Pi} \Gamma_t$. \text{Formally,}
\begin{align}\label{e:a12}
\begin{aligned}
\bvf_n= & (\bvf \cdot \n) \n=\left( \n \otimes \n \right) \bvf, \\
\bvf_{\pi}= & \bvf- \bvf_n =\mathbf{\Upsilon}_{\n}\bvf.
\end{aligned}
\end{align}

Next, $\bvf_{\pi}$ \blue{is further projected along tangential and lateral directions of the plane} $\mathit{\Pi}_{\xf}\Gamma_t$, \blue{giving rise to vector components $\bvf_t$ and $\bvf_l$}. In this study, the substrate is kept stationary and flat. Projected velocities represent the relative velocity of sliding\blue{, and after defining the projection tensor $\mathbf{\Upsilon}_{\tu}=\mathbf I-\tu\otimes\tu$, they} are given by
\begin{align}\label{e:a13}
\begin{aligned}
\bvf_t= & (\bvf_{\pi} \cdot \tu) \tu=\left( \tu \otimes \tu \right) \bvf_{\pi}\blue{=(\tu\otimes\tu)\mathbf \Upsilon_{\n}\bvf}, \\
\bvf_l= & \bvf_{\pi}- \bvf_t =\mathbf{\Upsilon}_{\tu} \bvf_{\pi}\blue{=\mathbf\Upsilon_{\tu}\mathbf \Upsilon_{\n}\bvf}. 
\end{aligned}
\end{align}

\begin{figure}[H]
    \centering
\includegraphics[width=8cm]{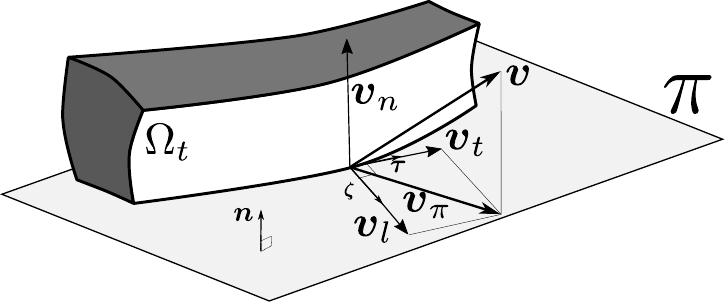}
 \caption{Projection of relative velocity field along substrate normal $\n$ and tangent plane $\mathit{\Pi}_{\xf}\Gamma_t$ for a limbless slender organism on the frictional substrate.}
   \label{fig: 1t}
\end{figure}

Motivated by the standard beam-based model for limbless locomotion \cite{gray64, munoz22, bijjm23}, we assume frictional anisotropy along the tangential ($\tu$) and lateral direction ($\bh$). Frictional forces induced at the body-substrate interface have been modelled in \blue{the} reference configuration with the following contact law
\begin{align}\label{e:a14}
\zt_o:=-\mu_t \bvf_t-\mu_l \bvf_l,
\end{align}
where, $\mu_t$ and $\mu_l$ are the coefficients of friction along tangential and lateral directions, respectively. Substituting Eq. \eqref{e:a12} and Eq. \eqref{e:a13} in Eq. \eqref{e:a14}, traction vector can be expressed as
\begin{align}\label{e:a15}
\zt_o=-\mathbf{B}(\tu,\n) \bvf
\end{align}
where frictional tensor $\mathbf{B}$ is defined as (recall that $\n$ and $\tu$ are orthonormal)
\begin{align}\label{e:a16}
\mathbf{B}(\tu,\n):= \left( \mu_l \mathbf{I}+(\mu_t-\mu_l) (\mathbf{I}-\mathbf{\Upsilon}_{\tu}) \right) \mathbf{\Upsilon}_n = \mu_l (\mathbf{I}-\n \otimes \n) + (\mu_t-\mu_l) (\tu \otimes \tu).
\end{align}

%

Many limbless organisms exert frictional anisotropy along the forward and reverse tangent direction. Considering this, we generalise frictional tensor $\mathbf{B}$ and assume \blue{also anisotropic frictional conditions with respect to forward and backward directions along} tangent vector $\tu$
\begin{align}\label{e:a18}
\mu_t = \frac{\mu_f+\mu_b}{2} +\left(\frac{\mu_f -\mu_b} {2} \right) c
\end{align}
where, $c= \text{sign}(\bvf \cdot \tu) \in \left\lbrace-1,1 \right\rbrace $ is a step function, $\mu_f$ and $\mu_b$ are the coefficient of friction along forward (tangent $\tu$) and reverse direction (opposite to $\tu$), respectively. In case, we have isotropy in the forward and reverse direction, then we impose $\mu_f=\mu_b$.

The above anisotropic friction law needs detection of the sign of velocity and hence introduces discontinuities in the dynamic simulation. As per standard procedure, following regularised/smoothed hyperbolic-tangent form has been used (see Fig. \ref{fig:1}b)
\begin{align}\label{e:a19}
c:=\tanh{\left(\frac{\bvf \cdot \tu}{\beta} \right)}
\end{align}
where $\beta \approx 0.1$ is a tolerance parameter of the Tanh-regularisation.

\begin{rem} \emph{Direction of limbless locomotion depends upon the degree of frictional anisotropy and direction of growth wave propagation. For undulatory locomotion with anisotropy along normal ($\mu_n$) and tangent direction (with isotropy along forward and reverse direction, $\mu_t=\mu_f=\mu_b$) limbless organism moves opposite to the direction of growth wave propagation (retrograde locomotion). Interestingly, frictional anisotropy along forward ($\mu_f$) and reverse direction ($\mu_b$) can generate either a prograde or retrograde locomotion i.e, motion along (prograde) or reverse (retrograde) to the direction of growth wave propagation. For instance, if a growth wave is propagating along a forward direction then frictional anisotropy $\mu_f<\mu_b$ generates a prograde locomotion whereas $\mu_f>\mu_b$ generates a retrograde locomotion.}
\end{rem}


\subsection{Viscous dissipation and balance laws}\label{s:compuLaws}

We neglect external body forces (gravity), but instead, we aim to regularise our problem by assigning forces proportional to the velocity at all points, in addition to the surface friction. This is equivalent to applying a body force proportional to velocity, $\zb_0=-\mu_0 \bvf$. In continuum setting, we hypothesise the existence of viscous energy dissipation $\psi^v:=\frac{\mu_o}{2} \bvf \cdot \bvf$ per unit reference volume and corresponding viscous energy/Rayleigh dissipation functional $\mathcal{W}^v:= \int_{\Omega_o} \psi^v d\Omega_o$. Then, non-conservative body force field $\zb_o$ expressed in the reference configuration reads
\begin{align}\label{e:a20}
\zb_o:= -\nabla_{\bvf} \psi^v=-\mu_o \bvf.
\end{align}

With the notion of internal energy functional, and non-conservative Rayleigh dissipation functional (viscous and frictional forces), potential energy functional can be constructed and for equilibrium, potential energy functional should attain its stationary value for all admissible variations in the deformation field (also known as Hu-Washizu Principle). Stationarity of the potential energy functional leads to the conservation of linear momentum and in the absence of inertial forces the local balance of linear momentum in the reference configuration reads \cite{wriggers08}
\begin{align}\label{e:a21}
\nabla_{\zx} \cdot \zP+\zb_o=\0\blue{, \qquad \forall \zx\in\Omega_o.}
\end{align}

%

At time t, we partition surface $\Gamma_o$ into the Dirichlet boundary $\Gamma_o^x$ and Neumann boundary $\Gamma_o^f$ such that $\Gamma^x \cap \Gamma_o^f=\emptyset$, $\Gamma_o^x \cup \Gamma_o^f=\Gamma_o$. Essential and natural boundary conditions are defined as
\begin{align}\label{e:a23}
\begin{aligned}
\xf(\zx, t)=\overline{\x} \qquad &\text{on } \Gamma_o^x, \\
\zP \hat{\N} =\zt_o \qquad &\text{on } \Gamma_o^f,
\end{aligned}
\end{align}
where, \blue{$\overline{\x}$} and $\zt_o$ are prescribed spatial position and traction field, respectively. At $t=0$, the spatial position is known and prescribed as $\xf(\zx, 0)=\x_0$.

Since the traction field is a function of spatial position and velocity, we consider both of them as a primary field and impose the \blue{following additional differential equation}
\begin{align}\label{e:a24}
\bvf -\dot{\xf}= \0, \qquad \text{in } \Omega_o,
\end{align}
along with the initial conditions $\xf(\zx, 0)=\x_0$, and $\bvf(\zx, 0)=\bv_0$.

Eq. \eqref{e:a21} and, Eq. \eqref{e:a24}, \blue{jointly} with Eq. \eqref{e:a23} form a system of non-linear \blue{BVP}. For arbitrary geometry and boundary conditions a close-form solution may not exist. The finite element method has been used to obtain the weak solution of the strong form of the BVP and is discussed in the next section.


\subsection{Spatial discretisation: two-field formulation}\label{s:compu2field}

In the absence of inertia forces, the strong form of the governing BVP forms a system of PDEs \blue{that we recast as}
\begin{align}\label{e:a25}
& \nabla_{\zx} \cdot \zP+\zb_o=\0, && \text{in } \Omega_o \times \mathcal{I} \\ \nonumber
& \bvf - \dot{\xf}=\0, && \text{in } \Omega_o \times \mathcal{I} \\ \nonumber
& \zP \hat{\N} =\zt_o, && \text{on } \Gamma_o^f \times \mathcal{I} \\ \nonumber
& \bc( \xf-\overline{\x},\bvf-\overline{\bv})=\0, && \text{on } \Gamma_o \setminus \Gamma_o^f \times \mathcal{I}
\end{align}

For variational formulation, we assign the function spaces for primary field $\xf, \bvf \in H^1(\Omega_o)$ and their variations $\bpf, \bqf \in H_o^1(\Omega_o)$. Then at any time $t \in \mathcal{I}$, we define energy functional $\mathcal{G} \in \mathbb{R} $ given by
\begin{align*}
\mathcal{G}(t,\xf, \bvf; \bpf, \bqf):= \int_{\Omega_o} \left( \bvf - \dot{\xf}\right) \cdot \bpf d\Omega_o-\int_{\Omega_o} \left(\nabla_{\zx} \cdot \zP +\zb_o \right) \cdot \bqf d\Omega_o.
\end{align*}

\blue{After using relation $\nabla_{\zx} \cdot (\zP^\mathsf{T} \bqf)=\zP : \nabla_{\zx} \bqf+(\nabla_{\zx} \cdot \zP) \cdot \bqf$, applying divergence theorem, and substituting Neumann boundary condition in Eq \eqref{e:a28},} the energy functional can be expressed as
\begin{align}\label{e:a28}
\blue{\mathcal{G}(t,\xf, \bvf; \bpf, \bqf)= \int_{\Omega_o} \zP : \nabla_{\zx} \bqf d\Omega_o - \int_{\Omega_o} \bqf \cdot \zb_o d\Omega_o - \int_{\Gamma_o} \bqf \cdot \zt_o d\Gamma_o + \int_{\Omega_o} \left( \bvf - \dot{\xf}\right) \cdot \bpf d\Omega_o.}
\end{align}

\blue{It then follows, that at any time $t \in \mathcal{I}$, the variational problem reduces to the following weak form: } 
\begin{align}\label{e:a29}
&\blue{Find\ \xf, \bvf \in H^1(\Omega_o)\ such\ that\ \mathcal{G}(t,\xf, \bvf; \bpf, \bqf) =0, \quad \forall \bpf, \bqf \in H_o^1(\Omega_o).}
\end{align}

Now we introduce a semi-discretisation\blue{, i.e., a spacial discretisation of fields $\xf, \bvf, \bpf, \bqf$ and $u$}. \blue{The reference configuration is discretised into \blue{$E$ finite elements $\Omega_o^e$} with $\Omega_o= \bigcup_{e=1}^{E} \Omega_o^e$. We assume that the growth field $u$ is a time-dependent known element-wise constant function}. Spatial positions $\xf$ and velocities $\bvf$ within \blue{$m$-noded element $e$ are respectively interpolated with the nodal values $\prescript{e}{}{\x}^a$ and $\prescript{e}{}{\bv}^a$ ($a=1,2,\dots,m$ with summation over $a$)} \cite{bonet97}
\begin{align}\label{e:a30}
\xf(\zx,t)= N^a (\zx) \prescript{e}{}{\x}^a(t), \quad \bvf(\zx,t)= & N^a (\zx) \prescript{e}{}{\bv}^a(t), \quad u(\zx,t)= \prescript{e}{}{\text{u}}(t)
\end{align}
where \blue{$N^a(\zx)$ are the interpolation functions}. Vectors $\prescript{e}{}{\bv}:=\{\prescript{e}{}{\bv}^1, \dots,\prescript{e}{}{\bv}^a, \dots \prescript{e}{}{\bv}^m \}$ and $\prescript{e}{}{\x}:=\{\prescript{e}{}{\x}^1, \dots,\prescript{e}{}{\x}^a, \dots \prescript{e}{}{\x}^m \}$ are elemental \blue{degrees of freedom}, with $\prescript{e}{}{\bv}^a=\{\prescript{e}{}{\text{v}}^a_x,\prescript{e}{}{\text{v}}^a_y, \prescript{e}{}{\text{v}}^a_z \}$, and $\prescript{e}{}{\x}^a=\{\prescript{e}{}{\text{x}}^a_x,\prescript{e}{}{\text{x}}^a_y, \prescript{e}{}{\text{x}}^a_z \}$. With the standard Galerkin procedure, the \blue{test functions within element $e$ is approximated resorting to the same interpolation}
\begin{align}\label{e:a31}
\bpf(\zx,t)= N^a (\zx) \prescript{e}{}{\bp}^a(t), \quad \bqf(\zx,t)= N^a (\zx) \prescript{e}{}{\bq}^a(t).
\end{align}


\blue{It then follows that the gradient of the test functions are given by $\nabla_{\zx} \bpf= \prescript{e}{}{\bp}^a (t) \otimes \nabla_{\zx} N^a$ and $ \nabla_{\zx} \bqf= \prescript{e}{}{\bq}^a (t) \otimes \nabla_{\zx} N^a$. After inserting the interpolation in \eqref{e:a30}-\eqref{e:a31} into the weak form \eqref{e:a29}, and from the arbitrariness of $\prescript{e}{}\bp^a$ and $\prescript{e}{}\bq^a$, the following system of non-linear Differential-Algebraic Equations  (DAE) are obtained,}
\begin{align}\label{e:a32}
\begin{aligned}
& \bigcup_{e=1}^{E} \prescript{e}{}\M (\prescript{e}{}{\bv}-\prescript{e}{}{\dot{\x}}) = \0\\
& \bigcup_{e=1}^{E}\left( \prescript{e}{}\g_{int}(\prescript{e}{}{\x},\prescript{e}{}{\bv},\prescript{e}{}u)- \prescript{e}{}\g_{ext}(\prescript{e}{}{\x},\prescript{e}{}{\bv},\prescript{e}{}u)\right)=\0 
\end{aligned}
\end{align}
\blue{where the elemental residuals $\prescript{e}{}\g_{int}(\prescript{e}{}{\x},\prescript{e}{}{\bv},\prescript{e}{}u)$ and $\prescript{e}{}\g_{ext}(\prescript{e}{}{\x},\prescript{e}{}{\bv},\prescript{e}{}u)$ associated to node $a$ and elemental mass matrix $\prescript{e}{}\M$ associated to nodes $a$ and $b$ are given by}
\begin{align}\label{e:gintext}
\begin{aligned}
\prescript{e}{}\g^a_{int}(\prescript{e}{}{\x},\prescript{e}{}{\bv},\prescript{e}{}u)&=\int_{\Omega_e} \zP \nabla_{\zx} N^a d\Omega_e
= \int_{\Omega_e} J_g \zP_e \zF_g^{-T} \nabla_{\zx} N^a d\Omega_e,\\
\prescript{e}{}\g^a_{ext}(\prescript{e}{}{\x},\prescript{e}{}{\bv},\prescript{e}{}u)&=\int_{\Omega_e} N^a \zb_o d\Omega_e+ \int_{\Gamma_e} N^a \zt_o d\Gamma_e, \\
 \prescript{e}{}\M^{ab}&=\int_{\Omega_e} N^a N^b d\Omega_e \mathbf I. 
\end{aligned}
\end{align}

\blue{In Eq.  \eqref{e:a32}, symbol  $\bigcup_{e=1}^{E}$ denotes standard FE assembling process of elemental vectors and matrices. By defining the global nodal position, $\x=\{\prescript{1}{}{\x}\ldots \prescript{E}{}{\x}\}$, equivalent global velocity and growth vectors, $\bv$ and $\bu$, respectively, and from the fact that the global mass matrix $\M$ is non-singular, equations in \eqref{e:a32} can be written as a system of index-1 semi-explicit DAE (provided $\nabla_{\bv} \g$ is non-singular at any time $t \in \mathcal{I}$)} \cite{ascher98}
\begin{align}\label{e:a33}
\begin{aligned}
\bv-\dot{\x}=\0 \\
\g(\x,\bv,\bu)=\0
\end{aligned}
\end{align}
along with the \blue{initial and} boundary conditions
\begin{align}\label{e:a34}
\begin{aligned}
&\blue{\x(0)=\x_0, \bv(0)=\bv_0}\\
&\bc( \x-\overline{\x},\bv-\overline{\bv})=\0 \quad \text{on } \Gamma_o \setminus \Gamma_o^f \times \mathcal{I}
\end{aligned}
\end{align}

In here, spatial position $\x$ act as a differential variable while spatial velocity $\bv$ is an algebraic variable.
\blue{In forthcoming sections we will make use of the centroid (or centre of mass) $\x_{cm}$ w.r.t inertial frame, which at any time $t$ can be computed from the FE discretisation as}
\begin{align}\label{e:a37}
\x_{cm}(t)=\frac{\int_{\Omega_o} \xf(\zx,t) d\Omega_o}{\int_{\Omega_o} d\Omega_o}=\frac{1}{\Omega_o} \bigcup_{e=1}^{E} \left(\int_{\Omega_e} N^a d\Omega_e \right) \prescript{e}{}{\x}^a= \mathbf{\Lambda} \x(t)
\end{align}
where $\mathbf{\Lambda}=\frac{1}{\Omega_o} \bigcup_{e=1}^{E} \prescript{e}{}{\Lambda}$, is the global mass distribution matrix and $\prescript{e}{}{\Lambda}^a:=\int_{\Omega_e} N^a d\Omega_e$ is the elemental mass distribution vector.


\section{Optimal control formulation: inverse problem}\label{s:OCP}

In this section, we shall consider a generic infinite-dimensional optimal control problem which seeks optimal control distribution $\bu(t)$ that minimises an objective functional ($\mathcal{J}: \mathbb{V}_{\x} \times \mathbb{V}_{\bv} \times \mathbb{V}_{\bu} \rightarrow \mathbb{R}$) subject to the constraints expressed in the form of a system of semi-explicit index-1 DAEs\blue{, with a similar structure to those in \eqref{e:a33}-\eqref{e:a34}},
\begin{align}\label{e:b1}
\min_{\x,\bv,\bu} & \ \mathcal{J}(\x,\bv,\bu) \\ \nonumber
\text{s.t.}&\ \bv - \dot{\x} =\0 && \text{(State ODE)}\\ \nonumber
&\ \g(\x,\bv, \bu)=\0 && \text{(Algebraic equation)}\\ \nonumber
&\ \bc( \x-\overline{\x},\bv-\overline{\bv})=\0 &&\text{(Boundary condition)}\\
&\blue{\x(0)=\x_0,\bv(0)=\bv_0} &&\blue{\text{(Initial condition)}}
\end{align}

\blue{We seek solutions} $\x \in W^{1,\infty}(\mathcal{I},\mathbb{R}^{n_1})$, $\bv \in L^{\infty}(\mathcal{I},\mathbb{R}^{n_1})$, and $\bu \in L^{\infty}(\mathcal{I},\mathbb{R}^{n_2})$ with $n_1$ and $n_2$, respectively, \blue{the number of state and control degrees of freedom resulting from FE discretisation. For clarity, dependence on time $t$ has been omitted in all these variables.} For the tracking-type optimal control problem under consideration, we aim to minimise the following objective functional
\begin{align}\label{e:b2}
\mathcal{J}(\x, \bu)=  \int^T_0 \left(r(\x)+q(\bu) \right) dt+\phi(\x(T)),
\end{align}
\blue{with $T$ a final (known) time, and with the} following quadratic form of the input $q(\bu)$, output $r(\x)$ and terminal cost $\phi(\x(T))$,
\begin{align}\label{e:b3}
r(\x)&:=\frac{1}{2} (\x_{cm}-\x_d)^\mathsf{T} (\x_{cm}-\x_d), \\ \nonumber
q(\bu)&:=\frac{\alpha}{2} \bu^\mathsf{T} \bu, \\ \nonumber
\phi(\x(T))&:=\frac{1}{2} (\x_{cm}(T)-\x_d)^\mathsf{T}. (\x_{cm}(T)-\x_d).
\end{align}

Vector $\x_{cm}$ is the centroid position defined in Eq. \eqref{e:a37}, $\x_d$ is a desired/target centroid position, $\alpha$ a regularisation parameter which penalises the amount of input control, and $\phi(\x(T))$ is the terminal cost\blue{. We recall that in our examples, and at any time $t$, each component of $\bu$ is an elemental growth.}

Since the \blue{DAE constraints need} to be satisfied at all time instants, we introduce time-varying \blue{Lagranges multipliers ($\bl$, $\bm$, $\bx$)} and define a Lagrangian functional $\mathcal{L}:\mathbb{V}_{\x} \times \mathbb{V}_{\bv} \times \mathbb{V}_{\bu} \times \mathbb{V}_{\bl} \times \mathbb{V}_{\bm} \times \mathbb{V}_{\bx} \rightarrow \mathbb{R} $ associated to Eq. \eqref{e:b1} as
\begin{align}\label{e:b4}
\mathcal{L}(\x,\bv, \bu; \bl, \bm, \bx)= & \int^T_0 \left(r(\x)+q(\bu) \right) dt+ \int^T_0 \left(\bl^\mathsf{T} \g(\x,\bv, \bu) + \bm^\mathsf{T} \left(\bv-\dot{\x} \right)\right) dt+\phi(\x(T))+\bx^\mathsf{T}\bc.
\end{align}

Using integration by parts, Lagrangian functional can be rearranged to
\begin{align}\label{e:b5}
\mathcal{L}(\x,\bv, \bu; \bl, \bm,\bx)= \int^T_0 \left(\mathcal{H}+\x^\mathsf{T} \dot{\bm} \right) dt-\bm(T)^\mathsf{T}\x(T)+\bm(0)^\mathsf{T}\x(0)+\phi(\x(T))+\bx^\mathsf{T}\bc,
\end{align}
\blue{where we have defined the control Hamiltonian as}
\begin{align}\label{e:b6}
\mathcal{H}(\x,\bv, \bu; \bm, \bl)=r(\x)+q(\bu)+\bl^\mathsf{T}\g(\x,\bv, \bu) +\bm^\mathsf{T}\bv.
\end{align}

First-order optimality \blue{conditions} can be obtained by setting the first variation of the Lagrangian w.r.t $\x, \bv, \bm, \bl, \bu$ equal to zero for all admissible field variations. For the vector field $\x$, the explicit form of the first variation of Lagrangian functional (Gateaux derivative) at $\x$ along the test vector field $\dd$ can be written as,
\begin{align}\label{e:b7}
\mathcal{D}_{\x} \mathcal{L} \cdot \dd= \lim_{\veps \to 0} \frac{\mathcal{L}(\x+\veps \dd,\bv, \bu; \bl, \bm,\bx)-\mathcal{L}(\x, \bv, \bu; \bl, \bm,\bx)}{\veps} =\frac{d}{d \veps} \mathcal{L}(\x+\veps \dd,\bv, \bu; \bl, \bm,\bx)\bigr\rvert_{\veps=0} = 0
\end{align}
and similarly for the remaining fields. Then, first-order stationary conditions also known as Karush-Kuhn-Tucker (KKT) conditions, result in the following system of DAEs (Adjoint, state and control equations)
\begin{align}\label{e:b8}
\begin{aligned}
\dot{\bm} &=-\nabla_{\x} \mathcal{H} && \text{(Adjoint equation)}\\ 
\0 &=-\nabla_{\bv} \mathcal{H} && \text{(Adjoint equation)}\\ 
\dot{\x} &=\nabla_{\bm} \mathcal{H} && \text{(State equation)}\\ 
\0 &=\nabla_{\bl} \mathcal{H} && \text{(State equation)}\\ 
\0 &=\nabla_{\bu} \mathcal{H} && \text{(Control equation)}
\end{aligned}
\end{align}
with \blue{initial and} boundary conditions
\begin{align}\label{e:b9}
\begin{aligned}
\blue{\x(0)=\x_0,\ \bv(0)}&\blue{=\bv_0,}\\ 
\bc( \x-\overline{\x},\bv-\overline{\bv})&=\0,\\ 
\bm(T)&=\nabla_{\x(T)} \phi(\x(T)).
\end{aligned}
\end{align}

\blue{Note that the algebraic relation $\nabla_{\bv}\mathcal H=\vect\mu+\nabla_{\bv}\vect g^T\vect\lambda=\0$ imposes implicitly the final condition $\bl(T)=-\left[ \nabla_{\bv}\g(T)\right]^\mathsf{-T} \bm(T)$}. The fact that $\text{ker}(\nabla_{\bv}\g)=\emptyset$ and the presence of regularised viscous/body forces ($\mu_o>0$) ensures this transformation. This condition is important for the well-posed time-discrete problem as shall be seen later in Section \ref{s:timedisc}. \blue{Equally, the relation $\nabla_{\bu}\mathcal H=\alpha\bu+\nabla_{\bu}\vect g^T\bl=\0$ imposes in turn final conditions for $\bu(t)$. The set of optimality conditions in \eqref{e:b8}} along with two-point boundary conditions \blue{in \eqref{e:b9} form so-called} Hamiltonian Boundary Value Problem (HBVP).

If we introduce a phase vector $\z:= \lbrace\x, \bm \rbrace $ and auxiliary vector $\bs:= \lbrace\bv, \bl, \bu \rbrace $, optimality conditions can be transformed into a compact form that reveals its geometric strcuture, 
\begin{align}\label{e:b10}
\dot{\z} &=\textbf{J}\nabla_{\z} \mathcal{H}, \\
\0 &=\nabla_{\bs} \mathcal{H}, \nonumber
\end{align}
where $\mathbf J \in \mathbb{SO}(2n)$ is the canonical symplectic matrix
\[
\mathbf{J}=\left[\begin{array}{rr}\0 & \mathbf I \\ -\mathbf I & \0\end{array}\right],
\]
\blue{which satisfies $\mathbf{J}^{-1}=\mathbf{J}^\mathsf{T}$ and $\vect a^\mathsf{T} \mathbf{J} \vect a =0, \forall \vect a\in\mathbb R^{2n}$ .} The above form reveals the hidden symplectic structure of the solution, even in the presence of a dissipative or forced system.

\begin{prop}\label{p:HPC}
For an autonomous dynamical system (no explicit time dependency), control Hamiltonian $\mathcal{H}$ is the first integral of the motion.
\end{prop}
\begin{proof}
For an autonomous dynamical system, we have that $\frac{\partial\mathcal H}{\partial t}=0$, and therefore the preservation of the total time derivative of the control Hamiltonian $\mathcal{H}\left(\z,\bs\right)$ follows directly from the Euler-Lagrange equations in Eq. \eqref{e:b10},
\begin{align*}
\dot{ \mathcal{H}}=\nabla_{\z} \mathcal{H}^\mathsf{T} \dot{\z}+\nabla_{\bs} \mathcal{H}^\mathsf{T} \dot{\bs} =\dot{\z}^T \mathbf{J} \dot{\z}+\0^\mathsf{T} \dot{\bs}=0.
\end{align*}
\end{proof}

\blue{Inspired by the conservation of control Hamiltonian or its symplecticity, we will resort to the Symplectic Euler (SE) integration scheme of the HBVP, as it will be explained in the next section}. Additionally, it has been shown \cite{ABJJM23} that SE results in stable trajectories even with the diminishing control regularisation parameter ($\alpha \rightarrow 0$) unlike mid-point or implicit Euler integrators which may introduce numerical oscillations.

In practice, the physical system has limited energy resources and control variables are often restricted to an admissible set i.e., $\bu \in \mathcal{U}_{ad}, \, \mathcal{U}_{ad}:=[\mathcal{U}_{min},\mathcal{U}_{max}]^{\blue{E}}$. In such cases, optimality condition w.r.t control variable results in a variational inequality without affecting state and adjoint equations. For a given $\bu \in \mathcal{U}_{ad}$, state and adjoint equations can be solved, and the objective functional $\mathcal{J}(\x,\bv,\bu)$ can be solely expressed as a functional of the control, i.e., $\hat{\mathcal{J}}(\bu):=\mathcal{J}(\x(\bu),\bv(\bu),\bu)$. Suppose, $\bu \in \mathcal{U}_{ad}$ is an optimal control. \blue{Then we must ensure that for a given a search direction $\dd$ and $\veps>0$} \cite{hinze09}
\begin{align}\label{e:b11}
\delta \hat{\mathcal{J}}(\bu)=\hat{\mathcal{J}}(\bu+\veps \dd)-\hat{\mathcal{J}}(\bu) \geq 0 \Rightarrow \lim_{\veps \to 0^+} \frac{\hat{\mathcal{J}}(\bu+\veps \dd)-\hat{\mathcal{J}}(\bu)}{\veps} =\int^T_0 \nabla_{\bu} \mathcal{H}^\mathsf{T} \dd \, dt \geq 0.
\end{align}

Above variational inequality is equivalent to the following conditions, for a given direction ${}^e\dd=\{0,\ldots,0,{}^ed,0\ldots,0\}$,
\begin{align*} 
\nabla_{\bu} \mathcal{H}^\mathsf{T} {}^e\dd \left\{\begin{array}{ll}
=0, & \text{if } \, \mathcal{U}_{min} \blue{<} {}^eu {<} \mathcal{U}_{max}\\
> 0, & \text{then } \, {}^eu=\mathcal{U}_{min} \\
< 0 , & \text{then } \, {}^eu=\mathcal{U}_{max}
\end{array}\right\},\blue{\ \forall {}^eu \in \mathcal{U}_{ad}}
\end{align*}

These conditions can be encoded into the following min-max projection operation to obtain the optimal growth distribution ($\bu \in \mathcal{U}_{ad}$)
\begin{align}\label{e:b13}
\bu^* &=\underset{\bu_c}{\mathrm{argmin} } \, \mathcal{H} \left( \x(\bu_c),\bv(\bu_c) ,\bm(\bu_c),\bl(\bu_c),\bu_c \right) \\
\bu &=\min \left( \mathcal{U}_\text{max},\max(\mathcal{U}_\text{min},\bu^*) \right)
\end{align}

\begin{rem} \emph{The Optimality conditions Eq. \eqref{e:b8}, can be written in an alternate form (provided $\nabla_{\bv} \g$ is invertible)}
\begin{align*}
& \dot{\x} = \bv && \text{(State equation)}\\
& \0 = \g (\x,\bv,\bu) && \text{(State equation)}\\
& \dot{\bm} =-\nabla_{\x} r(\x)+ \nabla_{\x}\g^\mathsf{T} \nabla_{\bv} \g^\mathsf{-T} \bm && \text{(Adjoint equation)}\\
& \0 = \nabla_{\bu} q(\bu) - \nabla_{\bu}\g^\mathsf{T} \nabla_{\bv}\g^\mathsf{-T} \bm && \text{(Control equation)}
\end{align*}
{\normalfont{with boundary conditions}}
\begin{align*}
\bc( \x-\overline{\x},\bv-\overline{\bv})=\0 \text{, and } \bm(T)=\nabla_{\x(T)} \phi(\x(T)).
\end{align*}
{\normalfont{If the following conditions are satisfied}}
\begin{enumerate}
\item $\nabla_{\x} \g$ and $\nabla_{\bv} \g$ are continuous and invertible mapping $\forall t \in \mathcal{I}$, and

\item $\alpha \in \mathbb{R}^+$ and $q(\bu)$ is quadratic in $\bu$,
\end{enumerate}
\emph{then, the above system is solvable for some $\bu \in \mathcal{U}_{ad}$ and the optimal control problem Eq. \eqref{e:b1}, admits at least one optimal tuplet $\left\lbrace \x^*, \bv^*, \bm^*, \bu^* \right\rbrace $.}
\end{rem}


\section{Time discretisation}\label{s:timedisc}

Time domain $\mathcal{I}=[0,T] \subset \mathbb{R}_0^{+}, T>0$, is uniformly partitioned into $N$ segments with a step size $\Delta t > 0$ such that $t_n=t_{n-1}+\Delta t$, $t_0=0$ and $t_N=T$. With generalised $\tau$-scheme, $\forall \, t_{n-\tau} \in [t_{n-1}, t_n]$ can be parametrised as $t_{n-\tau}:=\blue{(\tau-1) t_{n-1}+\tau t_n}$, $\tau \in [ 0,1 ]$. Then, the state variable $\x$ at time $t_{n-\tau_x}$ ($\tau_x \in [ 0,1 ]$) are evaluated as
\begin{align}\label{e:c1}
\x_{n-\tau_x} &=\blue{(1-\tau_x )\x_{n-1}+\tau_x \x_{n}}
\end{align}
and similarly for the remaining decision variables ($\bv, \bu, \bl, \bm$). Furthermore, state tangent matrices ($\K:=\nabla_{\x} \g$, and $\G:=\nabla_{\bv} \g$) and control tangent matrix ($\B:=\nabla_{\bu} \g$) at intermediate time point ($t_{n-\tau}$) and final time point ($t_N$) should be evaluated at $\lbrace \x_{n-\tau_x},\bv_{n-\tau_v},\bu_{n-\tau_u} \rbrace$ and $\lbrace\x_N,\bv_N,\bu_N\rbrace$, respectively. \blue{Appendix \ref{appendix:a} details the expressions of the tangent matrices}. For instance, control tangent matrix ($\B$) at the intermediate and final time point reads
\begin{align}\label{e:c2}
\B &= \nabla_{\bu} \g(\x_{n-\tau_x},\bv_{n-\tau_v},\bu_{n-\tau_u}) \\
\B_N &= \nabla_{\bu} \g(\x_N,\bv_N,\bu_N).
\end{align}

Now we restore to the numerical integration of the HBVP \blue{in Eq. \eqref{e:b8}-\eqref{e:b9}}. The generalised $\tau-$scheme results in the discrete system in Box \ref{b:TI}.

\begin{BOX}[!htb]
\begin{empheq}[box=\fbox]{align*}
&\quad \mathbf{\Lambda}^\mathsf{T}(\mathbf{\Lambda} \x_{n-\tau_x}-\x_d)+\K^\mathsf{T} \bl_{n-\tau_{\lambda}} + \frac{\bm_{n}-\bm_{n-1}}{\Delta t}=\0, \\
&\quad \G^\mathsf{T} \bl_{n-\tau_{\lambda}}+\bm_{n-\tau_{\mu}}=\0, \\
&\quad \bv_{n-\tau_v}-\frac{\x_{n}-\x_{n-1}}{\Delta t} = \0, \\
&\quad \g(\x_{n-\tau_x},\bv_{n-\tau_v},\bu_{n-\tau_u}) = \0, \\
&\quad \alpha \bu_{n-\tau_u}+ \B^\mathsf{T} \bl_{n-\tau_{\lambda}}=\0,\quad \text{for}\ n=1,\ldots, N, \\
&\quad \alpha \bu_N+ \B_N^\mathsf{T} \bl_N=\0 \\
&\text{with} \\
&\blue{\quad \x(0)=\x_0,\ \bv(0)=\bv_0},\\
&\quad \bc( \x-\overline{\x},\bv-\overline{\bv})=\0,\\
&\quad  \bm_N= \mathbf{\Lambda}^\mathsf{T}(\mathbf{\Lambda}\x_N-\x_d),\ \bl_N=-\G_N^\mathsf{-T} \bm_N.
\end{empheq}
\caption{\blue{Time Integration Scheme}}
\label{b:TI}
\end{BOX}

A generalised-$\tau$ integration algorithm can be used to design various integration schemes to accomplish desired numerical stability, accuracy and some time to preserve the integral of motion. For example, $\tau=0.5$ results in the mid-point scheme with second-order accuracy and preserve mechanical energy and linear momentum. Another combination, $\tau_x=\tau_v=0$ and $\tau_{\lambda}=\tau_{\mu}=\tau_u=1$ results in the symplectic Euler (SE) scheme which preserves the area in the phase-space (Liouville theorem) \cite{hairer02} and enhance the numerical stability of the optimal control problem \cite{ABJJM23}. In this work, we shall exploit the symplectic structure of the solution of HBVP and eventually use the \blue{SE} time integration scheme.

\section{Optimisation algorithm}\label{s:oalgo}

In this section, we shall discuss the computationally efficient solution procedure of the posed time-discrete HBVP. For initial admissible control history, one can integrate state DAE forward in time and after populating state history, adjoint DAE can be integrated backwards in time. After knowing state and adjoint histories, control history can be updated by taking admissible steps along the descent direction for objective functional. This procedure is popularly known as Forward Backward Sweep Method (FBSM) \cite{sharp21, lenhart07}. FBSM procedure should be repeated until the desired decrease in objective functional is achieved, or further iterations do not improve the objective functional significantly, or the norm of the control residue falls below a prescribed tolerance.


\subsection{Forward Backward Sweep Method}
In this work, Forward Backward Sweep Method (FBSM) is implemented in the following way \cite{lenhart07}:
\begin{enumerate}
\item \textbf{Initial guess}: Generate one admissible control trajectory \blue{$\bu^0\in\mathcal U_{ad}$} to start the algorithm.

\item \textbf{State time integration}: Given $\bu^{\blue{k}}$, \blue{initial condition $\x(0)=\x_0,\ \bv(0)=\bv_0$}, and boundary condition $\bc( \x-\overline{\x},\bv-\overline{\bv})=\0$, solve the following discrete non-linear problem at time $t_n$ ($n=1,2,\ldots, N$)
\begin{align}\label{e:d1}
\bv_{n-\tau_v}-\frac{\x_{n}-\x_{n-1}}{\Delta t} &= \0, \\
\g(\x_{n-\tau_x},\bv_{n-\tau_v},\bu_{n-\tau_u}) &= \0.
\end{align}

We resort to Newton-Raphson process for solving these equations: suppose, $\y_{n-1}:=\lbrace\x_{n-1},\bv_{n-1} \rbrace$ is known, then next state $\y_{n}:=\lbrace\x_{n},\bv_{n} \rbrace$ is iteratively updated $\y_n^{s+1}=\y_n^{s}+\Delta \y_n$ by solving following linearised system ($\y_n^1=\y_{n-1}$, $k=1,2,\dots$)
\begin{align}
\begin{bmatrix}
-\frac{1}{\Delta t}\textbf{I} & \blue{\tau_v}\textbf{I} \\
\blue{\tau_x} \K & \blue{\tau_v}\G \\
\end{bmatrix}_{\y_n^\mathsf{k} } \Delta \y_n= - \begin{bmatrix}
\bv_{n-\tau_v}-\frac{\x_{n}-\x_{n-1}}{\Delta t} \\
\g(\x_{n-\tau_x},\bv_{n-\tau_v},\bu_{n-\tau_u}) \\
\end{bmatrix}_{\y_n^\mathsf{k} } \nonumber.
\end{align}

\blue{Let us denote the resulting solution by $\y^k=\{\x^k$ and $\bv^k\}$.}

\item \textbf{Adjoint time integration}: Given \blue{$\bu^k, \x^k, \bv^s$} with the transversality conditions $\bm_N=\beta_3 \Lambda^\mathsf{T}(\Lambda \tilde{\x}_N-\x_d)$ and $ \bl_N=-\G_N^\mathsf{-T} \bm_N$, solve discrete adjoint linear system backward in time ($n=N,N-1,\ldots, 1$)
\begin{align}
\begin{bmatrix}
-\frac{1}{\Delta t}\textbf{I} & \blue{(1-\tau_{\lambda})}\K^\mathsf{T} \\
\blue{(1-\tau_{\mu})} \textbf{I} & \blue{(1-\tau_{\lambda})} \G^\mathsf{T} \\
\end{bmatrix} \begin{bmatrix}
\bm_{n-1}^{\blue{k}} \\
\bl_{n-1}^{\blue{k}} \\
\end{bmatrix}= - \begin{bmatrix}
\mathbf{\Lambda}^\mathsf{T}(\mathbf{\Lambda} \x_{n-\tau_x}^{\blue{k}}-\x_d)+ \blue{\tau_{\lambda}}\K^\mathsf{T} \bl_n^{\blue{k}} + \frac{1}{\Delta t} \bm_n^{\blue{k}}\\
\blue{\tau_{\lambda}}\G^\mathsf{T} \bl_n^{\blue{k}} + \blue{\tau_{\mu}}\bm_n^{\blue{k}} \\
\end{bmatrix} \nonumber.
\end{align}

\item \textbf{Control update}: Solve the control equation using Gradient Descent (GD) strategy. Directional derivative of the objective functional along $\bu$ can be expressed as ($\delta \bu=\veps \dd$, with $\dd$ a search direction)
\begin{align*}
\delta \hat{\mathcal{J}}(\bu)= \frac{\hat{\mathcal{J}}(\bu+\veps \dd)-\hat{\mathcal{J}}(\bu)}{\veps} \biggr\rvert_{\veps=0}= \int^T_0 \nabla_{\bu} \mathcal{H}^\mathsf{T} \delta \bu dt = 0.
\end{align*}
It can be noticed that maximum reduction of the objective function can be achieved if one moves opposite to the gradient direction $\nabla_{\bu} \mathcal{H}$ \cite{nachbagauer15}. Considering this, the discrete control residue at time $t_n$ is defined as $\br_n=\nabla_{\bu_n}\mathcal{H}$. Global residual vector $\br=\{\br_0,\ldots, \br_{N} \}^T$ at all time point reads
\begin{align*}
\br_{n-1}&= \alpha \bu_{n-\tau_u}+\B^\mathsf{T} \bl_{n-\tau_{\lambda}}, \ n=1,2, \ldots, N,\\
\br_N&=\alpha \bu_N+\B_N^\mathsf{T} \bl_N,
\end{align*}
and at each iteration $k$ we update the control variable as,
\begin{align}\label{e:updateu}
\bu^{k+1}=\bu^{k}+\theta^k \dd^k
\end{align}
with \blue{$\dd^k$} a search direction and $\theta^k$ the step-length at the $k$-th iteration and determined by appropriate \blue{line-search} scheme. \blue{If $||\br^{k+1}||<tol$, FBSM process is stopped. Otherwise,  step 2 is applied again.}

\end{enumerate}

\blue{In step 4, and in case that we require $\bu \in \mathcal{U}_{ad}:=[\mathcal{U}_{min},\mathcal{U}_{max}]^E$}, then we use a projection operation, where we update \blue{each elemental} control variable as
\[
\blue{{}^eu^{k+1}} =\min \left( \mathcal{U}_\text{max},\max(\mathcal{U}_\text{min},\blue{{}^e u^{k+1}}) \right)
\]

Search direction \blue{is initialised as $\dd^0:=-\br^{0}$} and is updated according to $\dd^{k+1}= -\br^{k+1}$. \blue{Further details of the line-search are given in the next subsection and in Appendix \ref{appendix:b}. Algorithm \ref{a:1}  summarises the FBSM process. }

\begin{algorithm}[!htb]
\begin{algorithmic} 
\Require $err=1; tol=10^{-3}, k=0, \bu_{min}, \bu_{max}, \bu^0$
\State $ \lbrace\blue{\x^0,\bv^0}\rbrace=\textbf{State}(\bu^0,\bc(\overline{\x},\overline{\bv}))$ \Comment {Solve state equation forward}
\State $ \lbrace\blue{\bm^0,\bl^0}\rbrace=\textbf{Adjoint}(\bu^0,\blue{\x^0,\bv^0},\bm_N,\bl_N)$ \Comment {Solve adjoint equation backward}
\State $ \br^0 = \nabla_{\bu} \hat{\mathcal{J}} (\bu^0,\blue{\x^0,\bv^0,\bm^0,\bl^0})$
\State $\dd^0 \leftarrow -\br^0$\Comment {Initial search direction}
\While{$err>tol$}
\State $\theta^k=\textbf{\blue{line-search}}(\hat{\mathcal{J}},\bu^k,\dd^k)$
\Comment {Line Search }
\State $\bu^{k+1} \leftarrow \bu^k+\theta^k\dd^k$
\If{$\bu_{min} >-\infty \, \text{or} \, \bu_{max}<\infty$}
\State $\bu^{k+1} \leftarrow \max \{\bu_{min}, \min \{\bu_{max},\bu^{k+1} \} \}$ \Comment {Projection step}
\EndIf
\State $ \lbrace\blue{\x^{k+1},\bv^{k+1}}\rbrace=\textbf{State}(\bu^{k+1},\bc(\overline{\x},\overline{\bv}))$ \Comment {Solve state equation forward}
\State $ \lbrace\blue{\bm^{k+1},\bl^{k+1}}\rbrace=\textbf{Adjoint}(\bu^{k+1},\blue{\x^{k+1},\bv^{k+1}},\bm_N,\bl_N)$ \Comment {Solve adjoint equation backward}
\State $ \br^{k+1} \leftarrow \nabla_{\bu} \hat{\mathcal{J}} (\bu^{k+1},\blue{\x^{k+1},\bv^{k+1},\bm^{k+1},\bl^{k+1}}) $ \Comment {Update residue}
\State $\dd^{k+1} \leftarrow -\br^{k+1}$\Comment {Search direction}
\State $err \leftarrow \max \{||\br^{k+1}||,||\bu^{k+1}-\bu^{k}|| \} $
\State $k \leftarrow k+1$
\EndWhile
\Ensure $\x^*,\bv^*,\bu^*,\mathcal{J}^* $ \Comment {Optimal solution}
\end{algorithmic}
\caption{\blue{Forward Backward Sweep Method combined with Gradient Descent line-search.}}
\label{a:1}
\end{algorithm}


\subsection{\blue{Line-search} scheme}

\blue{Line-search} schemes help in estimating an optimal/suitable step size $\theta^k>0$ along search direction $\dd^k$ such that \blue{$\hat{\mathcal{J}}(\bu^k+\theta^k \dd^k)<\hat{\mathcal{J}}(\bu^k)$}. There are mainly two ways to determine the \blue{line-search} parameter, namely exact and inexact methods\blue{, which correspond to exactly minimise $\hat{\mathcal{J}}(\bu)$, or just sufficiently reduce it \cite{nocedal06}}. In practical applications,
\blue{inexact backtracking schemes, such as Armijo and Wolfe method, are frequently used}. For large-size problems, solving forward and backwards is an expensive operation and best suited \blue{line-search} techniques must minimize this effort. The major drawback of the backtracking methods is that they often require many bisection operations to obtain a suitable $\theta_k$, which unnecessarily increases the computational burden on the FBSM. 

Unlike the backtracking method, Barzikai-Borwein \blue{line-search} scheme is a two-step method and usually accelerates the convergence rate of the GD method \cite{barzilai88}. The most attractive part of Barzikai-Borwein method is that it only requires one functional evaluation per iteration (see \blue{details} Appendix \ref{appendix:b}). In some cases, it is sufficient to take a small constant step length $\theta^k \approx \frac{1}{ ||\br^0||}$ along the search direction $\dd^k$, which should be reduced (generally halved) only if update leads to constraint violation.

\blue{We have adapted a stabilised version of Barzikai-Borwein algorithm to our update step in \eqref{e:updateu}. Appendix \ref{appendix:b} summarises the general and stabilised method, and also explains a methodology for ensuring positive steps. The details of the line-search update are given in Algorithm \ref{a:2} of Appendix \ref{appendix:b}. }


\section{Numerical examples}\label{s:numexp}

In this section, we show some numerical examples to validate the universality of the proposed computational framework. To solve the forward and optimal control problem, we developed an in-house Matlab code and employed it for the study. In the first example, \blue{we verify the growth formulation with a simple bending beam. In the second example,} a numerical solution to the forward problem is presented to simulate the distinct locomotion pattern adopted by limbless organisms. Additionally, the centre of mass sensitivity with the degree of substrate anisotropy and growth wave parameters viz. frequency and wave number has been studied. In the \blue{third} example, we seek the optimal growth distribution adopted by distinct gait which propels the centre of mass in a given amount of time to the desired target location. In both examples, 8 noded linear brick element with $3\times 3\times 3$ Gauss quadrature rule has been used. Total event duration \blue{$T$} is 4s with time-step size $\Delta t=0.05 s$ and viscous regularisation coefficient $\mu_o=10^{-3}$ is kept fixed in all numerical tests. Moreover, compressible Neo-Hookean model parameters viz. shear modulus $\mu=100$ and bulk modulus $\lambda=0.1 \mu$ have been used \blue{in our simulations}.

\subsection{\blue{Verification of growth formulation}}\label{s:ver}

In this section, we will benchmark the developed in-house code with the analytical solution. Let us consider the following Dirichlet BVP:
\begin{align*}
\min_{\xf} & \, \, \mathcal{U}(\xf, \buf) \\ \nonumber
\text{s.t.},\\   \nonumber
&\xf(\zx)=\overline{\x} 		&& \text{on } \Gamma_o^x \\ \nonumber
&\buf(\zx)=\bu_o &&\text{in } \Omega_o
\end{align*}
where growth $\buf$ is prescribed and the energy functional $\mathcal{U}$ is defined in Eq. \eqref{e:a3}. 
\begin{figure}[!htb]
\centering
\includegraphics[width=0.45\textwidth]{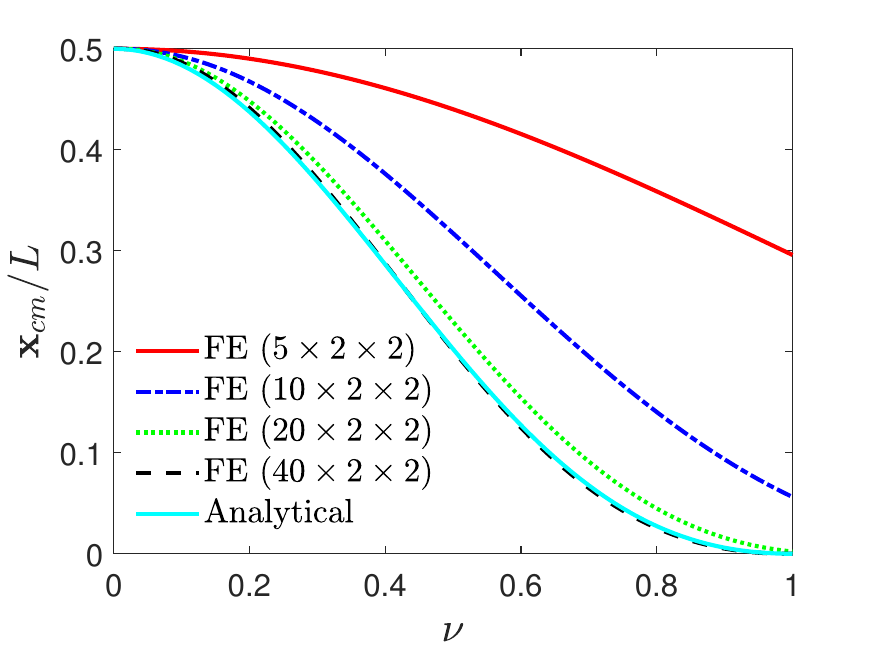} 
\includegraphics[width=0.46\textwidth]{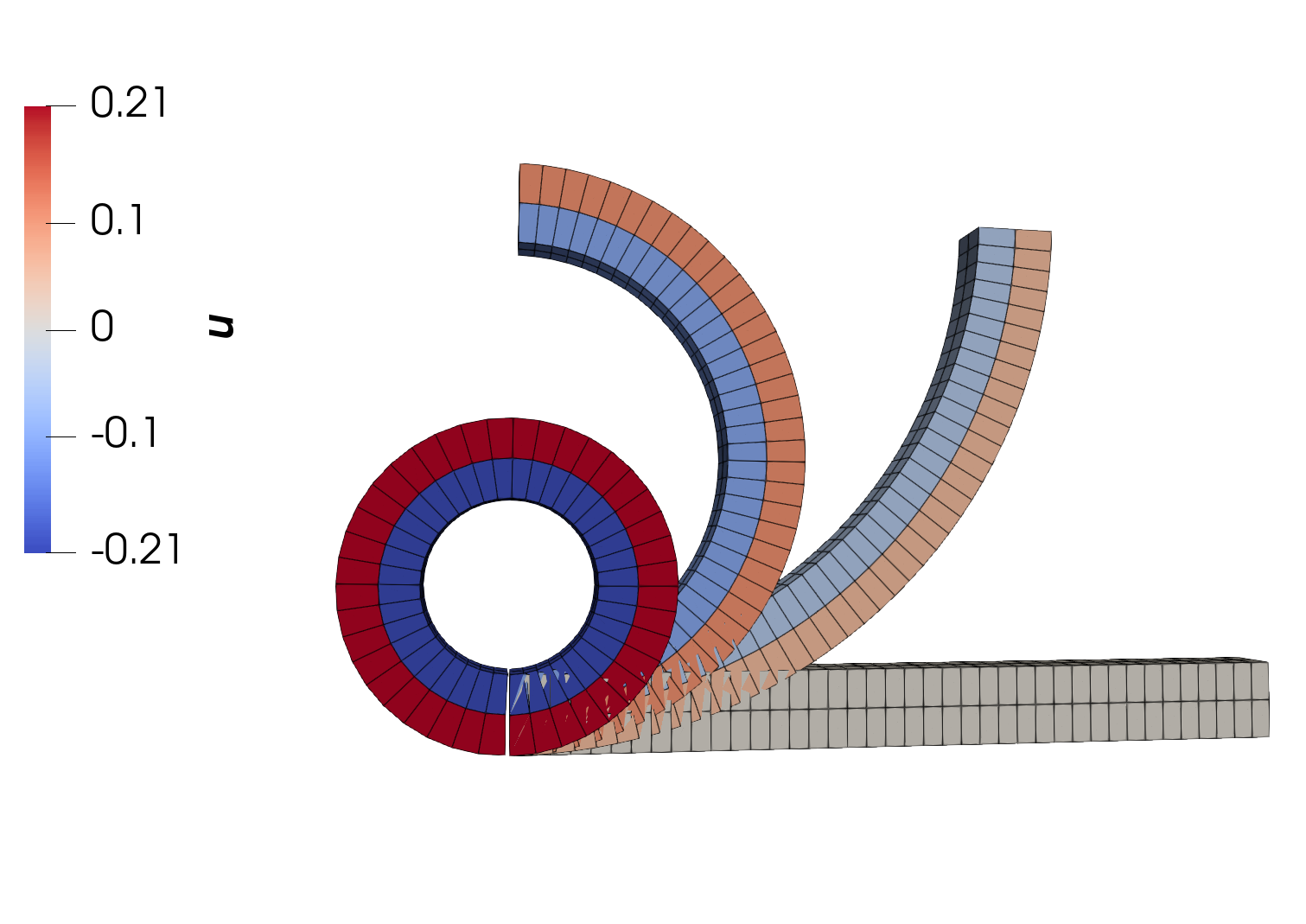}
\caption{Teatherd worm on frictionless substrate: (a) centroid $\x_{cm}$ evolution with growth $u$ and mesh refinement $n$, and (b) deformed shape (Ventral (Red) and Dorsal (Blue)).}
\label{fig: 10}
\end{figure}

We solve the above problem with the finite element method as described in the section \ref{s:compu2field}. Now we proceed to the closed-form solution of the cantilever beam subjected to antagonistically acting dipole forces. Let us consider\blue{a thin beam with geometry as shown in Fig. \ref{fig: 2}b, where length $L > > B$ and cross-sectional area is equal to $B^2$. We assume the muscles at top region} uniformly contract with growth $u=-u_o$, while bottom muscles uniformly expand with growth $u=+u_o$. 
\begin{align*} 
u = \begin{cases}
-u_o, & \text{if } \,    z \in [0, \, 0.5 B]\\
+ u_o, & \text{if } \,   z \in [-0.5 B, \, 0].
\end{cases}
\end{align*}

\blue{The fix Dirichlet boundary $\Gamma_o^x$ correspond to the left end, whre no longitudinal displacements are allowed}. Using classical beam theory, total stretch $\lambda$ can be decomposed into elastic, $\lambda_e=1+\varepsilon_e$, and growth  stretch $\lambda_g=1+u$, i.e., $\lambda=\lambda_e \lambda_g$. With linear approximation, total strain ($\varepsilon=- \kappa z$, $\kappa$ is principle curvature), decomposes into the sum of growth strain $\varepsilon_g=u$ and elastic strain $\varepsilon_e$, i.e., $\varepsilon=\varepsilon_g + \varepsilon_e$. Uniaxial stress for linear elastic beam in terms of flexural rigidity $D$ reduces to $\sigma=\frac{12 D}{B^3} (\varepsilon-\varepsilon_g)$. Then, at any cross-section, resultant moment $M$ and growth moment $M_g$ reads \cite{taber20}
\begin{align}\label{e:m2}
M &= - \int_{0}^{B} \int_{-B/2}^{B/2} \sigma z \,  dz dy= D \kappa + M_g \\
M_g &= \frac{12 D}{B^3} \int_{-B/2}^{B/2} u z  dz=\frac{12 D}{B^3} \left( \int_{-B/2}^{0} u_o z dz + \int_{0}^{B/2} (-u_o) z dz \right) =-\frac{3 D u_o}{B} 
\end{align}

Since beam is under self-equilibrated moments ($M=0$), curvature $\kappa$ at a point along the neutral-axis reads
\begin{align}\label{e:m4}
\kappa = - \frac{ M_g }{D}=\frac{3 u_o}{B}.
\end{align}

One can observe that curvature $k=3 u_o/B$ at any point on the neutral-axis is constant and represents a state of pure bending. With Eq. \eqref{e:m4}, $u_o=\frac{2 \pi \nu B}{3L}\approx 0.21 \nu$ \blue{where}  $\nu=0.25, 0.5,1$ represents the growth required to bend the beam into a quarter circle, semi-circle and a full circle, respectively. Then, the horizontal component of the worm centroid, $x_{cm}$, can be written as \cite{bijjm23}
\begin{align}\label{e:m5}
x_{cm} =\frac{2}{\kappa^2 L}  \sin^2 \left( \frac{\kappa L}{2}\right)  .
\end{align}

Theoretical results show that the horizontal displacement of the centroid is independent of beam material and should carry forward to the non-linear material model provided additive decomposition of strain remains valid (compressible Neo-Hookean solid with high value of shear modulus $\mu \approx 100$ and $\lambda \approx 0.1 \mu$). We have validated these findings by discretising the geometry $L \times B \times B$=$10 \times 1 \times 1$ with the finite elements ($n \times 2 \times 2$) and studied the convergence behaviour with grid refinement along the longitudinal direction ($n=5, 10, 20, \text{and } 40$). As depicted in Fig. \ref{fig: 10}a, centroid trajectories with $20$ and $40$ elements are consistent with the analytical result. Fig. \ref{fig: 10}b shows three snapshots of the initially straight beam transforming into a quarter circle ($\nu=0.25$), semi-circle ($\nu=0.5$) and full circle ($\nu=1$), which justifies the pure bending effect.


\subsection{Locomotion on the soft substrate: forward dynamics}

Lets us consider a slender limbless organism on a flat interface with rest length $L$ and cross-sectional area $B^2$, with $B=m L$. In the undeformed state, the Ventral-Dorsal Skeletal Muscles (VDSM) system, responsible for the muscle's active response, is assumed to be running parallel to the worm body length $L$. Essentially, we assume the origin of the distinct locomotion pattern emerges from the rhythmic contraction-extension activity of the VDSM and which constitutes a contractile dipole with antagonistic or synergistic polarity. For example, the antagonistic action of VDSM may result in undulatory or inching gait exhibited by \blue{\emph{C. elegans}} and Caterpillar, respectively. On the contrary, the synergistic action of VDSM may result in crawling gait found in larvae (maggot) or earthworms. Fig. \ref{fig: 2} shows the possible dipole polarity at various cross-sections running parallel to the worm body length. \blue{However, the adoption} of different gait strongly \blue{depends} on the degree of substrate frictional anisotropy. In this example, we shall focus on the generation of distinct gait pattern, their characteristics and the role of frictional anisotropy. We have mainly considered three distinct gait patterns generally adopted by the limbless organism on the soft substrates viz. undulatory, crawling and inching. Throughout, we discretised the cuboid geometry with $20 \times 2 \times 2$ finite elements along the length $L=10$,  and cross-sectional area $B^2=m L \times m L $. We use cross-sectional shrinkage factor $m=0.05$ and $m=0.1$ for inching gait and undulatory/crawling gait, respectively. For inchworms, a portion of the front and backward body part $W$ is constrained to be in contact with the substrate.
\begin{figure}[!htb] 
\centering
\subfigure[$\text{\blue{\emph{C. elegans} (undolutary)}}$]{\includegraphics[width=0.55\textwidth]{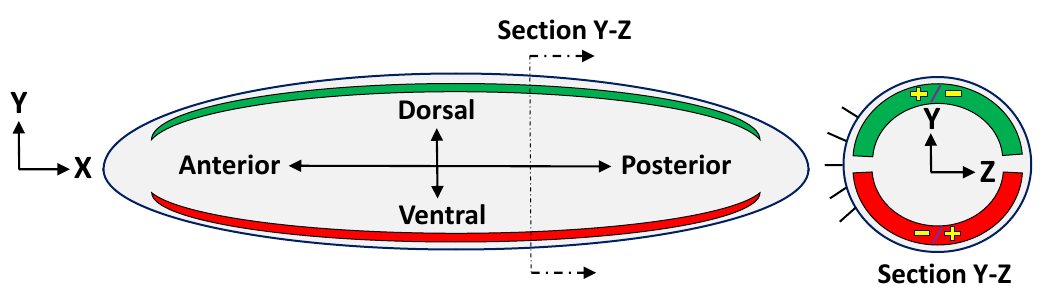}}
\subfigure[$\text{\blue{\emph{C. elegans}} model}$]{\includegraphics[width=0.35\textwidth]{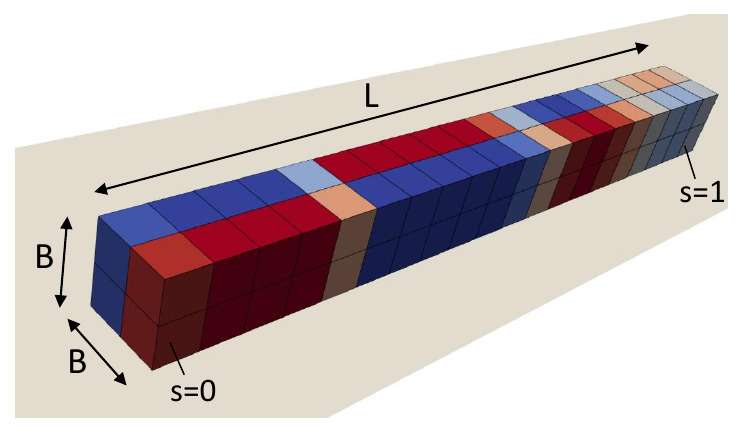}}
\subfigure[$\text{Larvae \blue{(crawling)}}$]{\includegraphics[width=0.55\textwidth]{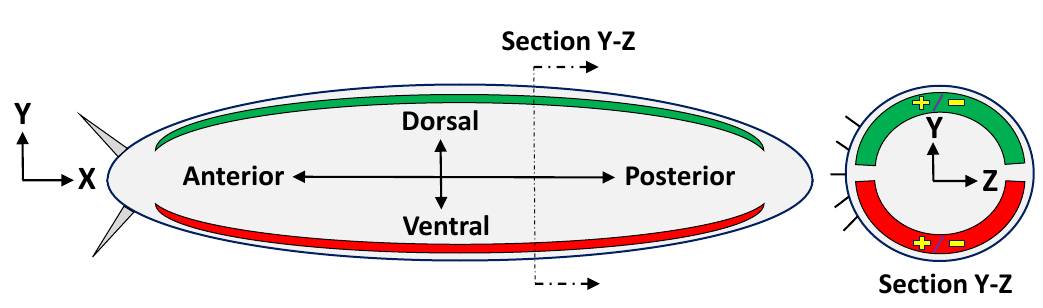}}
\subfigure[$\text{Larvae model}$]{\includegraphics[width=0.35\textwidth]{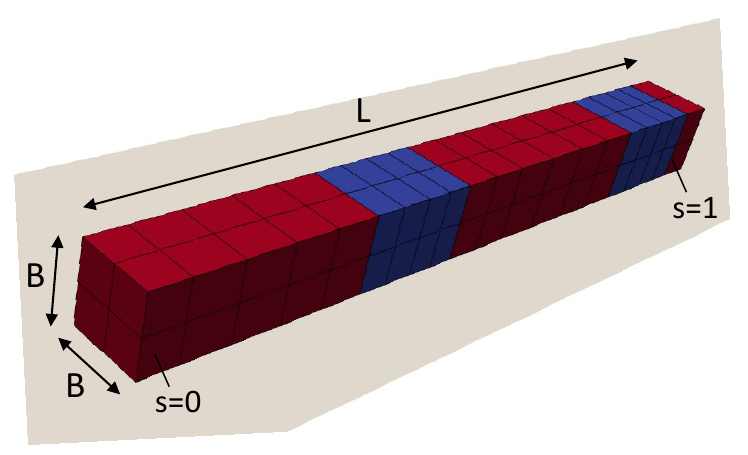}}
\subfigure[$\text{Caterpillar \blue{(inching)}}$]{\includegraphics[width=0.55\textwidth]{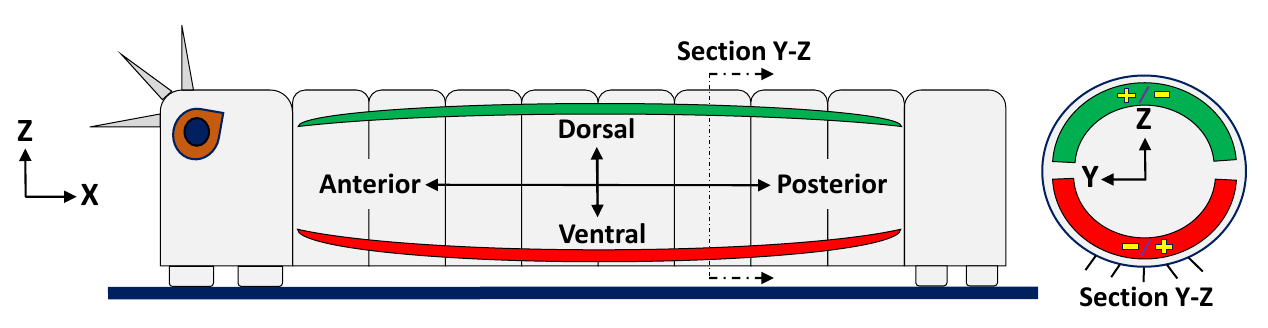}}
\subfigure[$\text{Caterpillar model}$]{\includegraphics[width=0.35\textwidth]{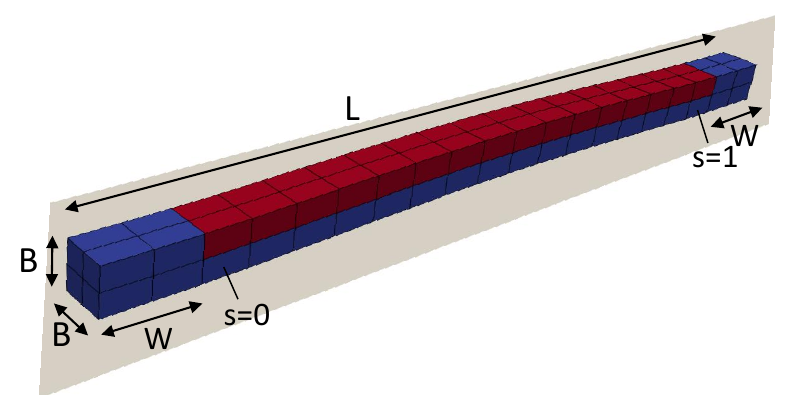}}
\caption{Left: Dorsal and ventral muscles arrangement in: (a) \blue{\emph{C. elegans}}, (c) Larvae, and (e) Caterpillar. Right: Contractility patterns in corresponding models.}
\label{fig: 2}
\end{figure}

Lets us begin with an undulatory limbless worm on a soft substrate with antagonistic action of laterally arranged VDSM. The arrangement of VDSM exerts muscular force in a specific fashion and induces frictional anisotropy with more frictional along the lateral direction ($\mu_l$) compared to the tangential direction ($\mu_t$). We parametrise the worm muscles length with a measure $s \in [0,1] $ and assume all muscles bundles inside the ventral and dorsal chamber at a given cross-section undergoes exactly equal and opposite amount of growth (Fig. \ref{fig: 2}a). In numerical implementation parameter $s$ is evaluated at the centroid of finite element approximating the undeformed worm, which makes elemental growth only a function of time. Considering this, the internal growth distribution $u(s,t)$ originated from the coordinated action of VDSM and can be expressed as a progressive wave. Then, growth disturbance in the dorsal and ventral chamber at any section $s$ \blue{and} time $t$ reads
\begin{align*}
u(s,t) &= + u_o \sin{(2 \pi f t + 2 \pi \gamma s)} && \text{(Ventral chamber)}\\
u(s,t) &= -u_o \sin{(2 \pi f t + 2 \pi \gamma s)} && \text{(Dorsal chamber)} 
\end{align*}
where $u_o \in \mathbb{R}$ is the amplitude of wave, $f$ the frequency, and $\gamma$ the wave number.

For crawling limbless locomotion, laterally arranged VDSM act synergistically and introduce frictional anisotropy along forward ($\mu_f$) and reverse ($\mu_b$) direction of advancement. The presence of setae or segmented body structure introduce anchoring phenomenon are a common source for breaking substrate frictional isotropy. Assuming all muscle bundles inside the ventral and dorsal chamber at a given cross-section undergo exactly the same amount of growth (Fig. \ref{fig: 2}c). Then, growth disturbance $u(s,t)$ in dorsal and ventral chamber at any section $s$ \blue{and} time $t$ reads
\begin{align*}
u(s,t) &= + u_o \sin{(2 \pi f t + 2 \pi \gamma s+\pi)} && \text{(Ventral chamber)}\\
u(s,t) &= + u_o \sin{(2 \pi f t + 2 \pi \gamma s+\pi)} && \text{(Dorsal chamber)} 
\end{align*}

For inching limbless locomotion, vertically arranged VDSM acts in an antagonistic fashion and introduce frictional anisotropy along forward ($\mu_f$) and reverse ($\mu_b$) direction of advancement (Fig. \ref{fig: 2}e). Unlike the previous gait cycle confined on the substrate, inching motion results in out-of-plane deformations. Furthermore, the shape of the inching gait is unaltered irrespective of locomotion speed as observed in Caterpillar. For lucidity, Caterpillar motion is decomposed in the following four stages (see Fig. \ref{fig: 2}e)
\begin{enumerate}

\item Contraction of ventral muscles and expansion of corresponding dorsal muscles, \blue{resulting in out-of-plane bend shape.}

\item Bending induces interface frictional forces which increase with the muscles contraction. 

\item Anchoring of the anterior body part results in a high frictional interface while airlifting the posterior body portion gives one stroke of inching motion. \blue{This differential friction tends to slide the anterior portion backwards and the posterior portion forward, but with a net forward motion.}

\item Gradually relaxation of ventral muscles and dorsal muscles while anchoring the posterior body part and forward airlifting of the anterior body part. This constitutes another stroke of the inching cycle and should be repeated continuously by modulating the frequency of stroking.

\end{enumerate}

\begin{figure}[!htb]
\centering
\subfigure[$\text{Undulatory}$]{\includegraphics[width=0.45\textwidth]{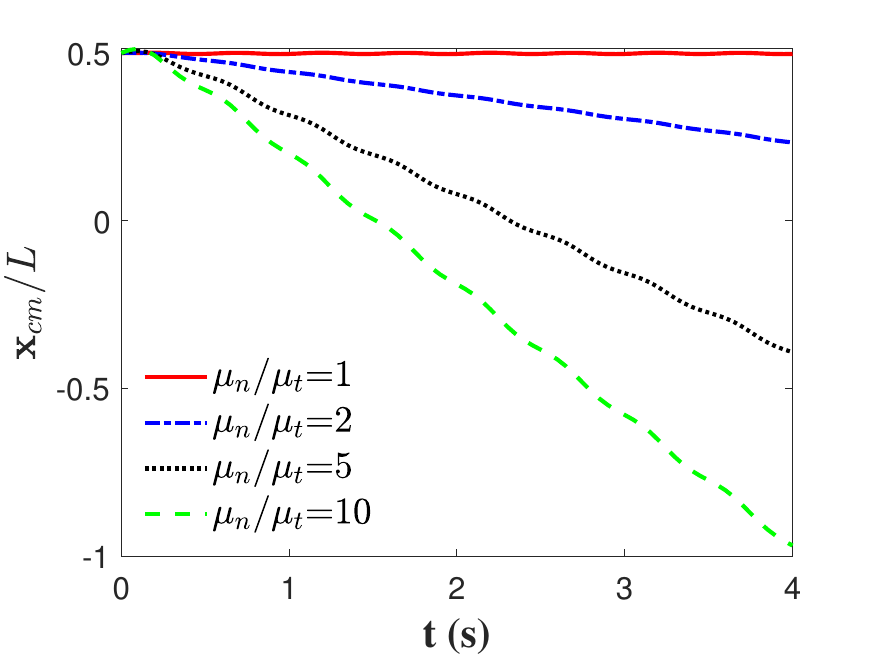}}
\subfigure[$\text{Crawling}$]{\includegraphics[width=0.45\textwidth]{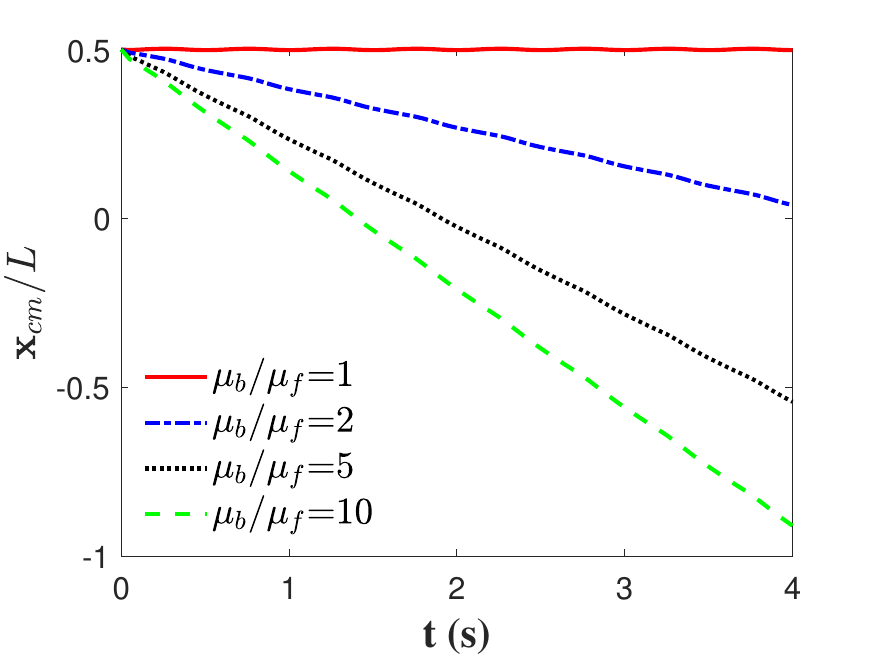}}
\subfigure[$\text{Inching}$]{\includegraphics[width=0.45\textwidth]{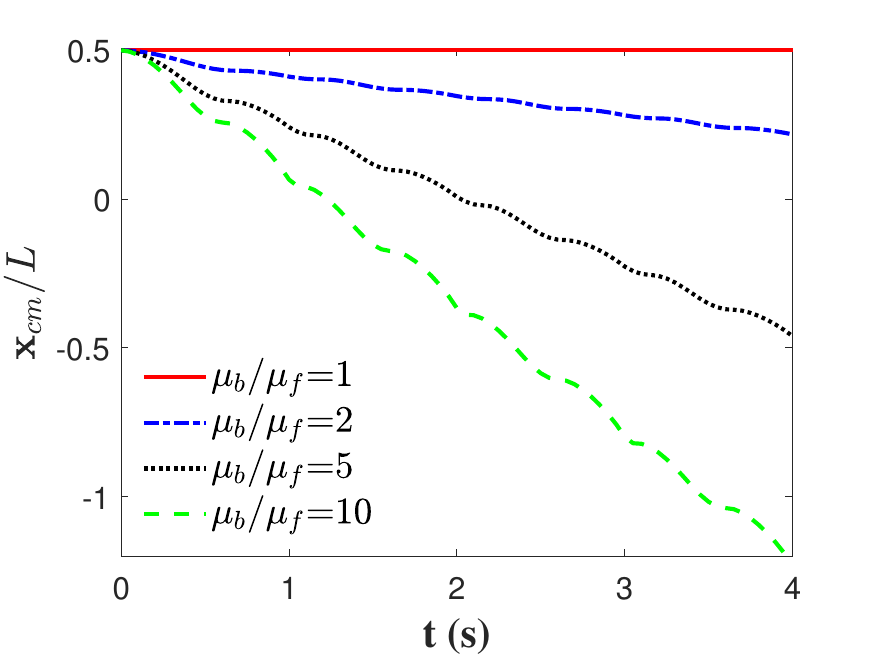}}
\subfigure[Anisotropic and Isotropic substrate]{\includegraphics[width=0.45\textwidth]{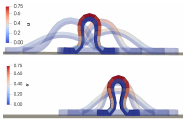}}
\subfigure[$\text{Anisotropic substrate}$]{\includegraphics[width=0.9\textwidth]{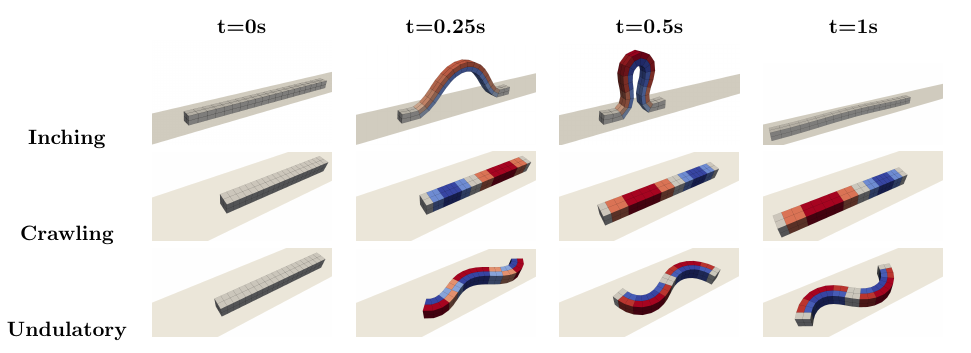}}
\caption{Worm centroid evolution with given muscles actuation and varying degree of frictional anisotropy: (a) undulatory gait (\blue{\emph{C. elegans}}), (b) crawling gait (larvae), (c) inching gait (caterpillar), (d) caterpillar on anisotropic (top with $\frac{\mu_b}{ \mu_f}=10$) and isotropic (bottom with $\frac{\mu_b}{ \mu_f}=1$) substrate, and (e) animation of distinct locomotion pattern for 1-second duration (periodic onwards)}
\label{fig: 3}
\end{figure}

The above four stages form a very complicated locomotion pattern and one has to ensure non-penetration constraint between a worm-substrate interface with periodic transition/switching between Dirichlet and Neumann boundary conditions. For computational efficiency, we penalise backwards sliding motion by enforcing high frictional conditions, whereas forward motion is appreciated by providing negligible frictional resistance i.e., $\mu_b >\mu_f$. 

Assuming all muscle bundles inside the ventral and dorsal chamber at a given cross-section undergo exactly equal and opposite amounts of growth (Fig. \ref{fig: 2}e). Then, growth disturbance $u(s,t)$ in dorsal and ventral chamber at any section $s$ and time $t$ can be approximated as
\begin{align}\label{e:z3}
\begin{aligned}
u(s,t) &= - u_o | \sin{(n \pi f t)} | \sin{(2 \pi \gamma s)} && \text{(Ventral chamber)}\\
u(s,t) &= + u_o | \sin{(n \pi f t)} | \sin{(2 \pi \gamma s)} && \text{(Dorsal chamber)} 
\end{aligned}
\end{align}
where \blue{$n$ is number of inching cycle/strokes and we fix the wave number to $\gamma=0.5$.}\red{A PLOT SHOWING $u(s,t)$ as a function of $s$ for two times (two cycles), AND ANOTHER PLOT WITH $\mu_b$ and $\ mu_t$ SEEMS NECESSARY}

It is important to note that for limbless locomotion the internal growth disturbance is often a progressive wave and frequently adopted by undulatory and crawling gait. On the contrary, inching gait generates the growth distribution as standing waves and can be formally interpreted from Eq. \eqref{e:z3}. In limbless locomotion, the net displacement of the centroid is of prime importance and a key parameter to evaluate the efficiency of locomotion. \blue{This} evolution of centroid is significantly influenced by the magnitude of frictional anisotropy, and growth wave characteristics ($u_o,f, \gamma$).

We begin the analysis with the influence of the degree of frictional anisotropy on the centroid horizontal displacement. For a undulatory gait with growth wave characteristic $u_o=0.3, f=1 \,\text{Hz},\gamma=1 $, centroid displacement is plotted \blue{in Fig. \ref{fig: 3}a for different ratios $\mu_n /\mu_t\in[1,10]$ }. It can be seen that $\mu_n /\mu_t =1$ represents an isotropic frictional state and the centroid undergoes no net horizontal displacement\blue{, while} an increase in the degree of anisotropy leads to an increase in the net displacement of the worm centroid. Physically, ratio $\mu_n /\mu_t \approx 2$ represents a low frictional interface and swimming is the preferred mode whereas ratio $\mu_n /\mu_t \approx 10$ represents a high frictional interface such as agar \cite{fangb10}. Similarly, for crawling gait, wave characteristic $u_o=0.3, f=2\text{Hz},\gamma=1 $, centroid displacement is plotted \blue{ in Fig. \ref{fig: 3}b for the same interval of ratios $\mu_b /\mu_f\in[1,10]$}. As expected, isotropic substrate ($\mu_b =\mu_f$) results in no net displacement of the centroid and with increasing $\mu_b /\mu_f$ ratio, there is a corresponding increase in the net displacement of the centroid. \blue{Same conclusions can be drawn for inching gait in Fig. \ref{fig: 3}c, where we used $u_o=0.75, f=1 \,\text{Hz}, n=4$ and $\gamma=0.5$ to mimic the caterpillar-like gait.} For visualisation purposes, \blue{isotropic and non-isotropic} scenarios have been depicted in Fig. \ref{fig: 3}d and highlighted shape shows the body posture after one inching stroke. It can be concluded that the degree of frictional anisotropy plays an important role in regulating limbless locomotion on soft substrates.

Next, we analyse the sensitivity of growth wave shape $\gamma$ and frequency $f$ on the net horizontal displacement of the centroid in undulatory and crawling gait. Numerical experiments were performed for the design space $\gamma \times f =[0.25, 2.5] \times [0.25, 2]$ with a fixed amplitude $u_o=0.3$ and a fixed degree of anisotropy \blue{$\mu_n /\mu_t=10$ and  $\mu_b /\mu_f=10$}. It can be observed \blue{in Fig. \ref{fig: 4}a} that for undulatory locomotion \blue{and for a given waveform,} increasing frequency of undulation maximises the centroid displacement. However, for a given frequency, there is always a particular waveform which results in the maximum horizontal displacement of the centroid. \blue{Similar trends are found for the crawling gait results shown in Fig. \ref{fig: 4}b. However, in this case, the region with maximum centroid displacements correspond to waveforms $\gamma \in [0.25, 1]$}. One can conclude that to achieve a particular centroid displacement there are large combinations of waveform and frequencies and all three gait can reach the same destination by modulating these wave characteristics. These findings raise a few important questions such as (a) optimal undulatory gait is retrograde or prograde locomotion, (b) for similar geometry and similar physiology conditions, which gait is energy efficient? and (c) how does the organism size influence the locomotion efficiency? In the next example, we will address these questions and draw some important conclusions.

\begin{figure}[H]
    \centering
    \subfigure[$\text{Undulatory}$]{\includegraphics[width=0.45\textwidth]{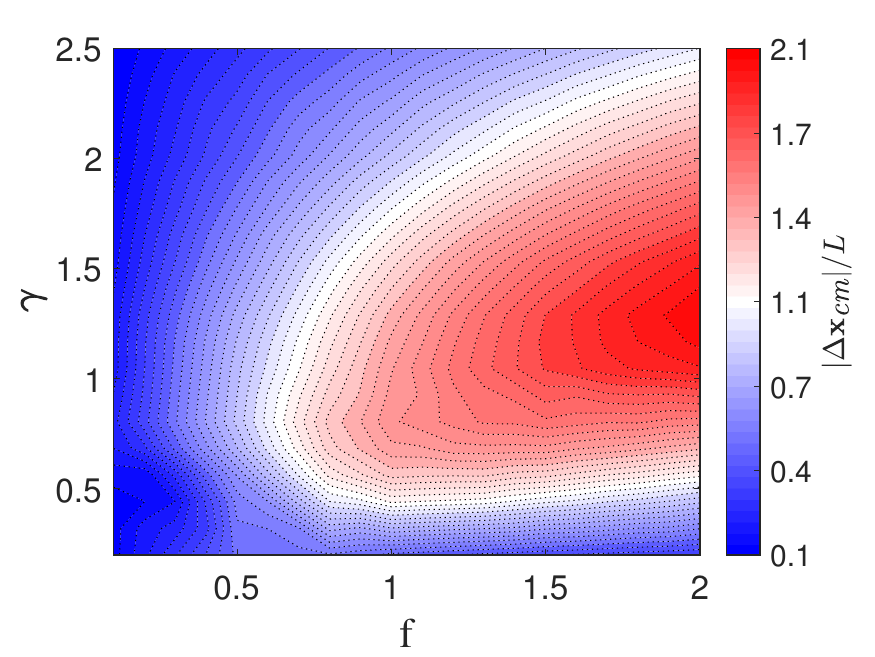}} 
     \subfigure[$\text{Crawling}$]{\includegraphics[width=0.45\textwidth]{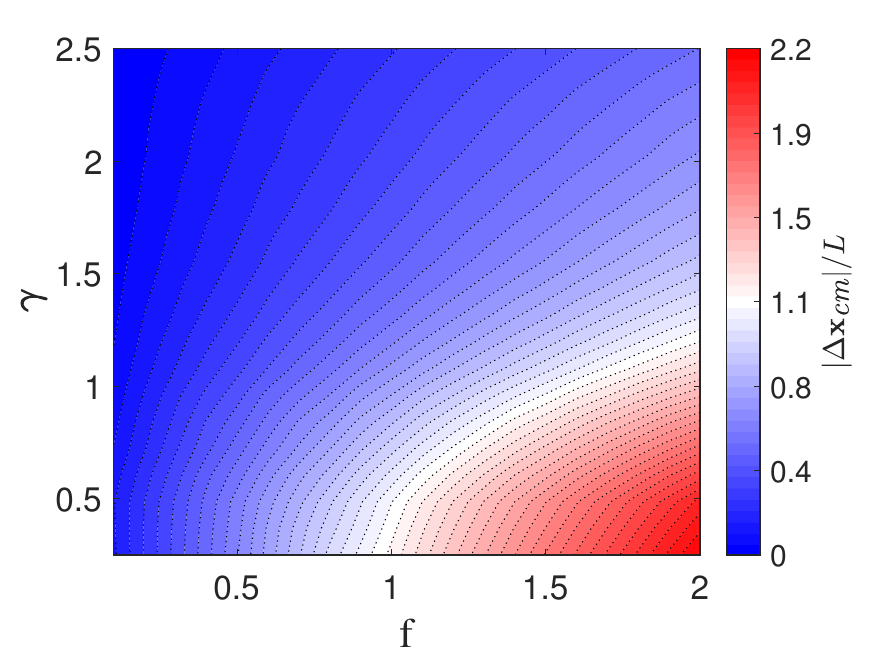}}	
    \caption{Worm centroid evolution with given degree of frictional anisotropy ($\frac{\mu_b}{ \mu_f}=10$) and varying muscles actuation frequency $f$ and shape $\gamma$: (a) undulatory gait (\blue{\emph{C. elegans}}), and (b) crawling gait (larvae).} 
   \label{fig: 4}
\end{figure}


\subsection{Optimal locomotion strategies}

We first examine the optimal locomotion trajectories adopted by limbless organisms with three distinct gaits. We are interested in the internal growth distribution ($\bu$) which propel worm centroid ($\x_{cm}$) from the initial state $\x_{cm}(0)=(5,0,0)$ to the desired location $\x_{cm}(T)=\x_d=(1,0,0)$ in fixed time duration ($T=4s$). Three FE models have been formed for undulatory, crawling and inching limbless locomotion. For fair comparison, we assign same material properties ($\mu=100, \lambda=0.1 \mu$), same geometry ($L \times B \times B=10 \times 1 \times 1$) with FE grid ($20 \times 2 \times 2$), and active growth is restricted to \blue{$u_o=0.3$}. Degree of frictional anisotropy \blue{ is set to $\mu_n /\mu_t=10$, and $\mu_b /\mu_f=10$ for} undulatory and crawling gait, respectively. However, for the inching gait to resemble the actual caterpillar locomotion, we assign frictional coefficient values $\mu_b=1$ and $\mu_f=0$. It is important to note that only one-half of FE cells (ventral chamber) growth participates in the optimisation process since the other half (dorsal chamber) is constrained to be antagonistically or synergistically related to it. Moreover, we assume at any cross-section all ventral muscles undergo exactly the same amount of growth and hence the size of the growth vector just reduces to the number of FE cells along the longitudinal direction (in this example 20 growth degrees of freedom at any time $t \in \mathcal{I}$).

\begin{figure}[!htb] 
\centering
\includegraphics[width=0.45\textwidth]{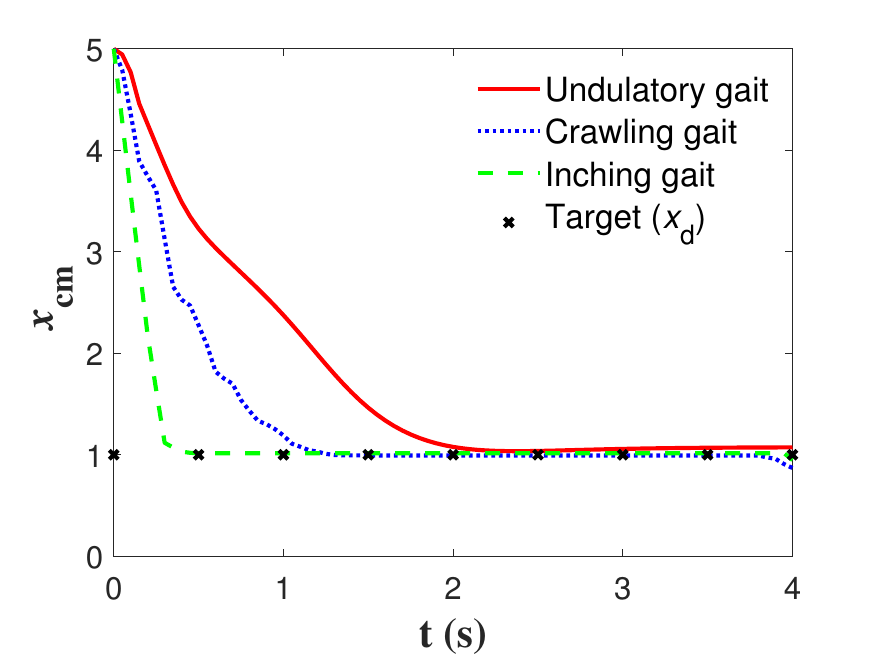} 
\includegraphics[width=0.45\textwidth]{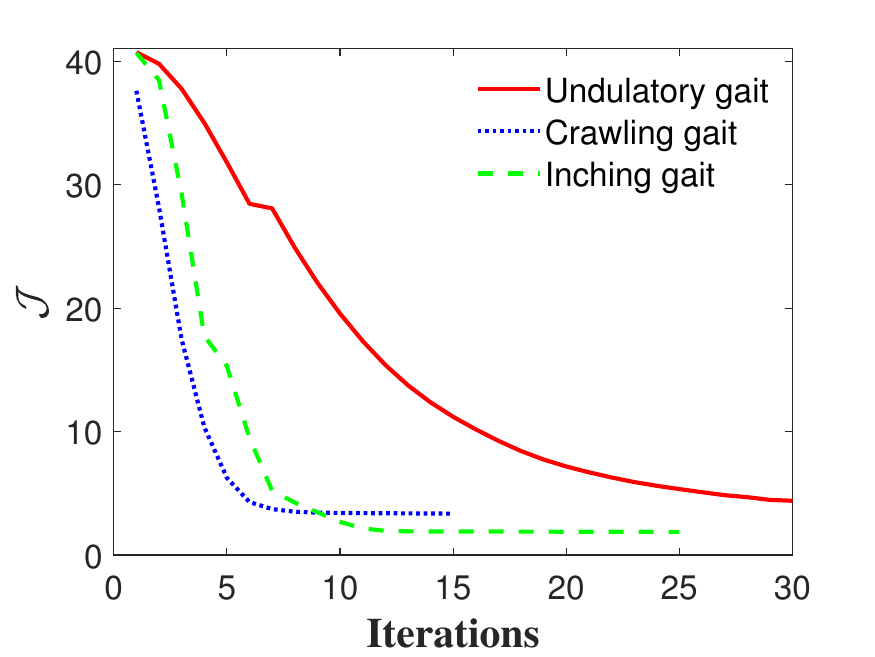}\\
\centerline{(a)\hspace{40ex}(b)}

\includegraphics[width=0.45\textwidth]{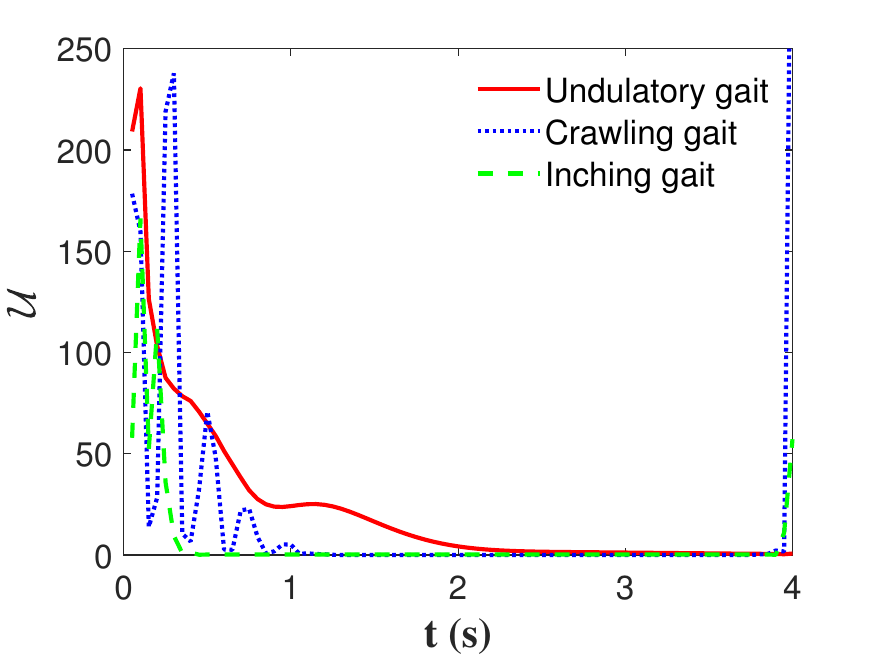}
\includegraphics[width=0.45\textwidth]{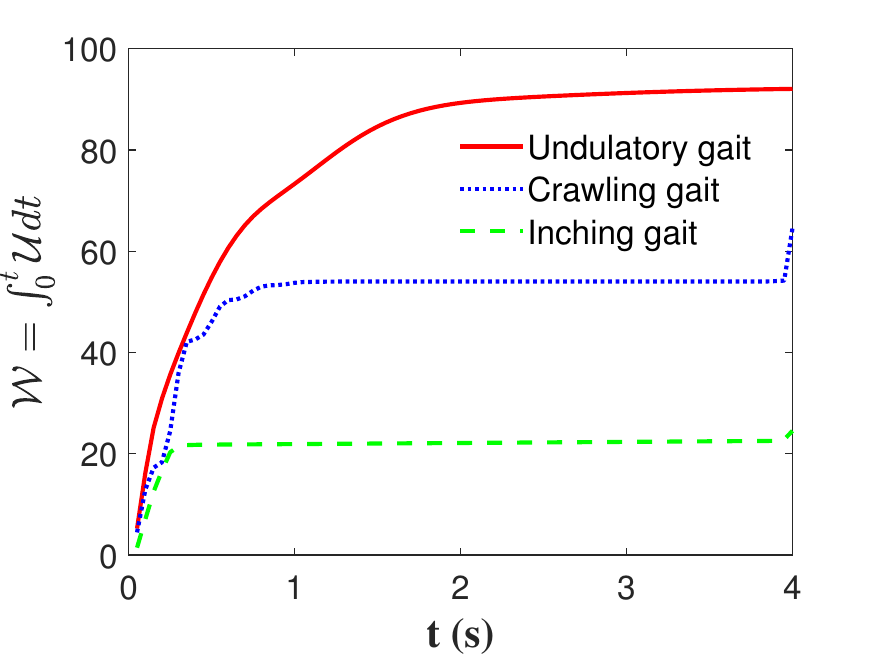}\\
\centerline{(a)\hspace{40ex}(b)}

\caption{Optimal locomotion in a limbless organism: (a) centre of mass ($\x_{cm}$), (b) cost functional ($\mathcal{J}$), (c) internal energy ($\mathcal{U}$), and (d) total energy spent ($\mathcal{W}$).}
\label{fig: 5}
\end{figure}

FBSM algorithm is initiated by assigning small growth $\bu^0=0.01 \mathbf{1}$. Growth distribution $\bu$ is confined to the space of admissible set $\mathcal{U}_{ad}=[-0.3,0.3]$ and $\mathcal{U}_{ad}=[-0.3,0]$ for undulatory/crawling and inching gait, respectively. The control regularisation parameter \blue{is fixed to $\alpha=10^{-3}$. The} Symplectic Euler time integration scheme ($\tau_x=\tau_v=0$, and $\tau_{\lambda}=\tau_{\mu}=\tau_u=1$ in Box \ref{b:TI}) has been used to solve the system of the discrete equation as described in Section \ref{s:oalgo}. For undulatory gait, the Barzikai-Borwein \blue{line-search} scheme has been used with $\theta_{max}=0.25$. For crawling and inching gait it was sufficient to assign a constant value to the \blue{line-search} parameter i.e., $\theta^k=\frac{1}{ ||\br^0||}$.

The simulation outcome for three gait cycles is plotted in Fig. \ref{fig: 5}. Among the three gaits, the inching gait \blue{attained} the desired system state in the least amount of time ($0.5 s$) while the undulatory gait took the longest time ($2 s$), and crawling duration ($1 s$) fall in the intermediate of both gaits (See Fig. \ref{fig: 5}a). Similar patterns have been observed in terms of total energy spent to attain the desired system state. It can be concluded that for similar physical conditions, the inching gait is most efficient followed by the crawling gait and the undulatory gait is \blue{the least} efficient\blue{, as Fig. \ref{fig: 5}d shows}. For completeness, we have shown the evolution of the objective functional \blue{as a function of the FBSM iterations in Fig. \ref{fig: 5}b,} and internal \blue{$\mathcal U$ and total spent energy $\mathcal W$ with time duration in Fig. \ref{fig: 5}c-d}.

Next, we plotted the optimal internal growth distribution predicted by the optimisation algorithm. For undulatory gait, growth distribution is moving backwards while the centroid is propelling in the forward direction, \blue{as Fig. \ref{fig:6}a shows}. This confirms that the optimal undulatory gait results in retrograde locomotion. Alteration of tangent vector (curvature) as a function of parameter $s$ is shown in Fig. \ref{fig:6}b. Curvature contour suggests that the worm body initially takes two sinusoidal waveforms which gradually reduce to one and a half waveforms. Fig. \ref{fig:6}c reveals that in crawling gait, growth disturbance is moving in the forward direction and in phase with the forward displacement of the centroid. This suggests that the optimal crawling gait results in prograde locomotion. Finally, the inching gait optimal growth distribution is shown in Fig. \ref{fig:6}d. Initially, the caterpillar generates a sharp rectangular growth pulse and in subsequent strokes, it converges to a half-sinusoidal pulse and then vanishes as the system attains its desired state.

\begin{figure}[!htb] 
\centering
\includegraphics[width=0.45\textwidth]{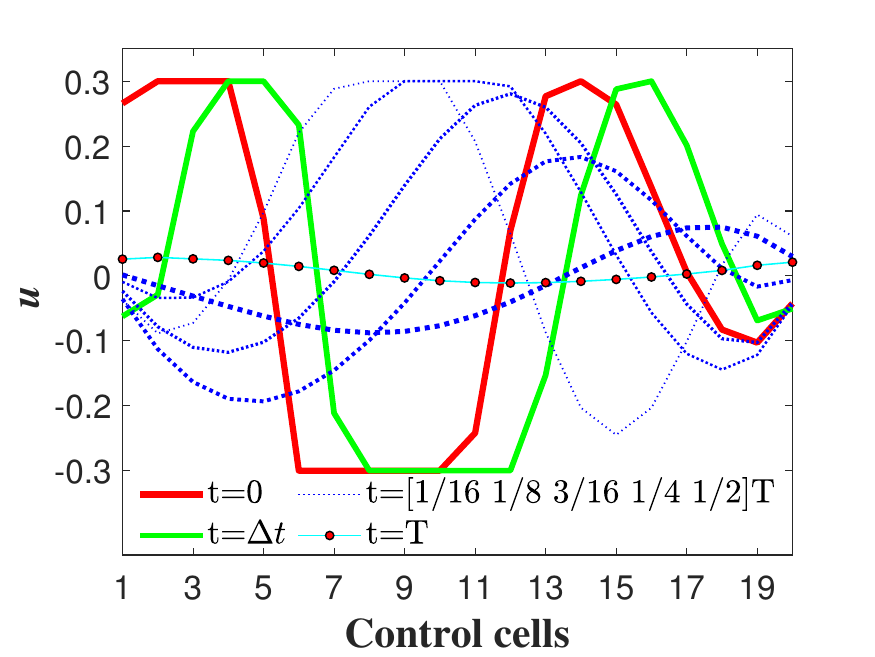} 
\includegraphics[width=0.45\textwidth]{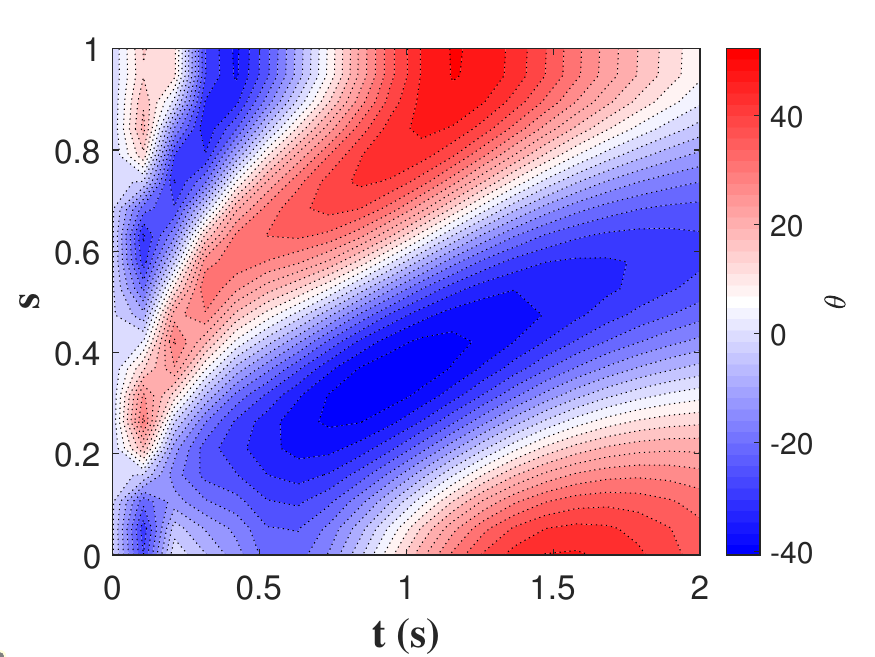}
\includegraphics[width=0.45\textwidth]{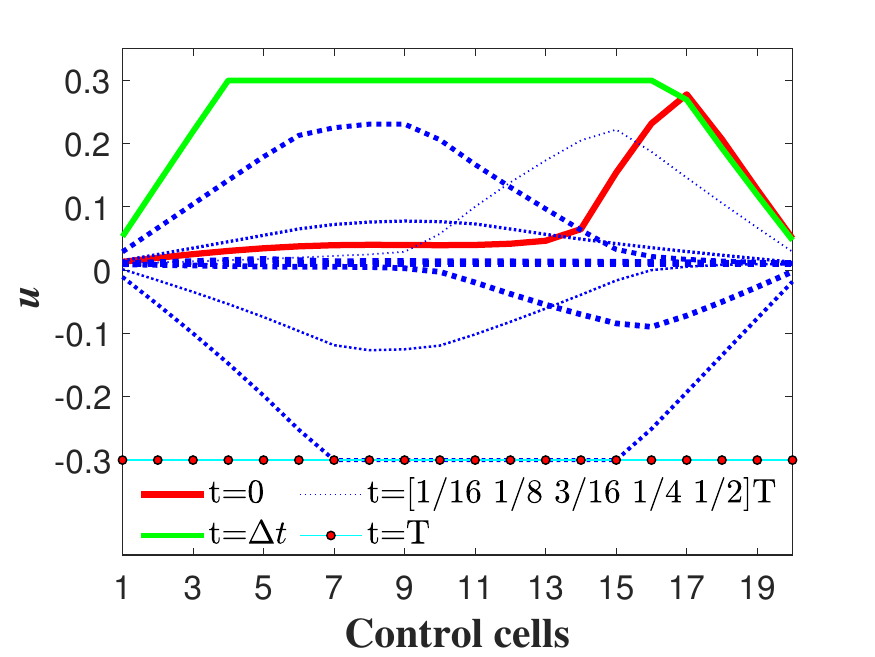}
\includegraphics[width=0.45\textwidth]{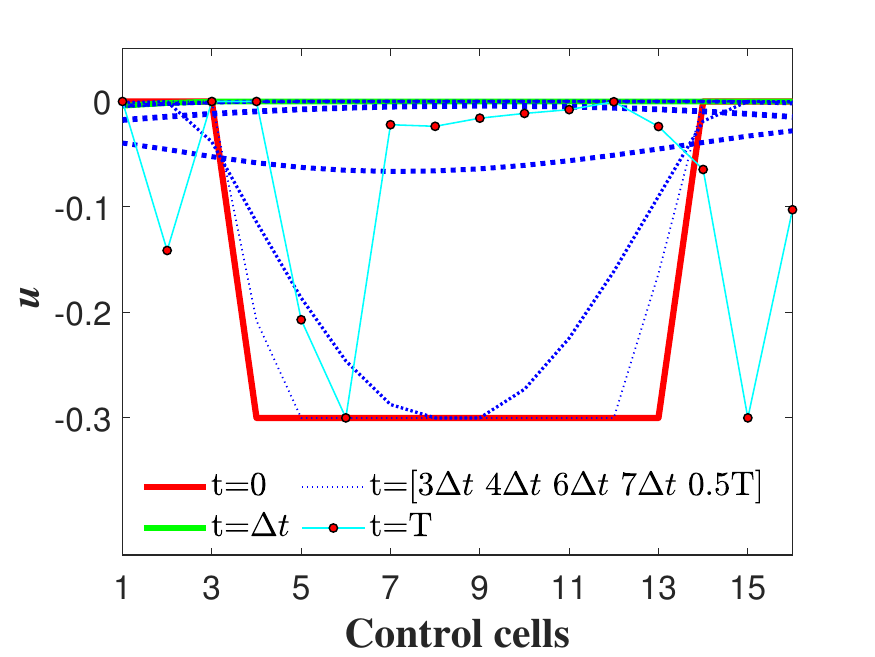}
\caption{Optimal locomotion in a limbless organism: (a) undulatory growth distribution, (b) undulatory body curvature variation, (c) crawling growth distribution, and (d) inching growth distribution.}
\label{fig:6}
\end{figure}

\subsection{Size effects}\label{s:size}

The influence of size effects on the limbless organism locomotion efficiency has been studied by doubling the initial volume i.e., $L \times B \times B=20 \times 2 \times 2$. We seek internal growth distribution which propels scaled worm centroid from initial state $\x_{cm}(0)=(10,0,0)$ to the final state $\x_{cm}(T)=\x_d=(5,0,0)$ in the fixed time duration ($T=4s$). We keep all other parameters unchanged compared to the unscaled version of this problem.
\begin{figure}[!htb]
\centering
\includegraphics[width=0.4\textwidth]{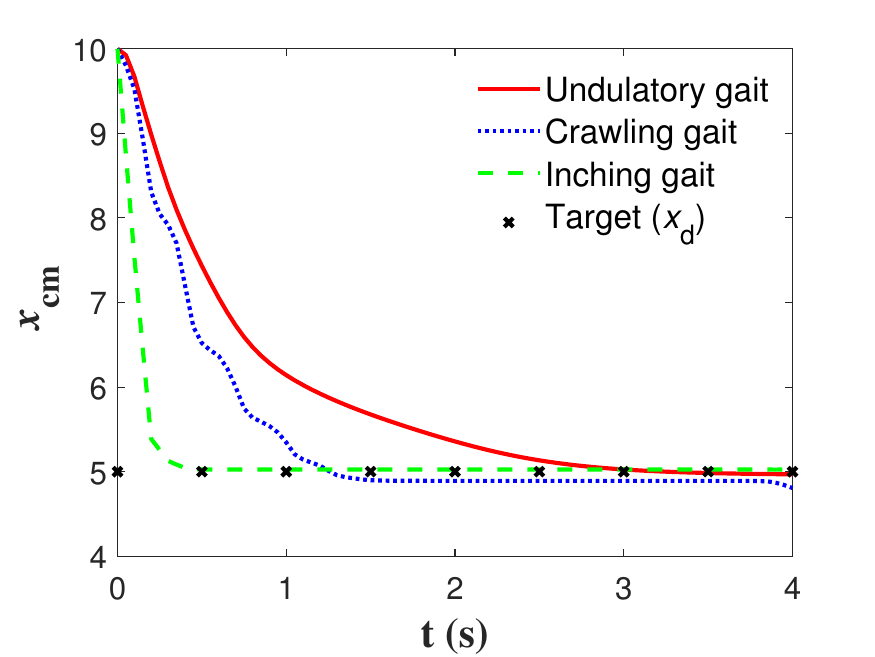} 
\includegraphics[width=0.4\textwidth]{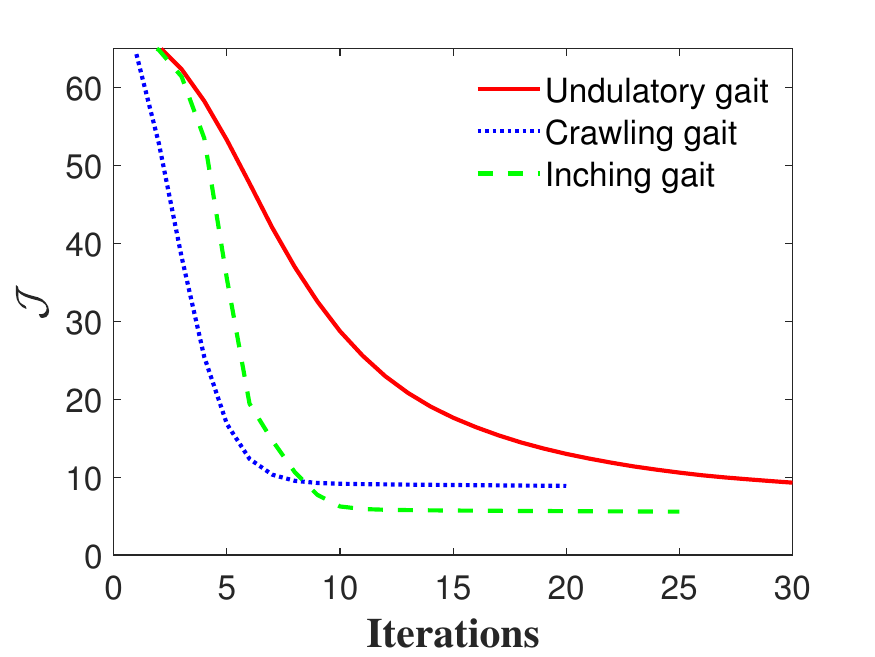}\\
\centerline{(a)\hspace{40ex}(b)}

\includegraphics[width=0.4\textwidth]{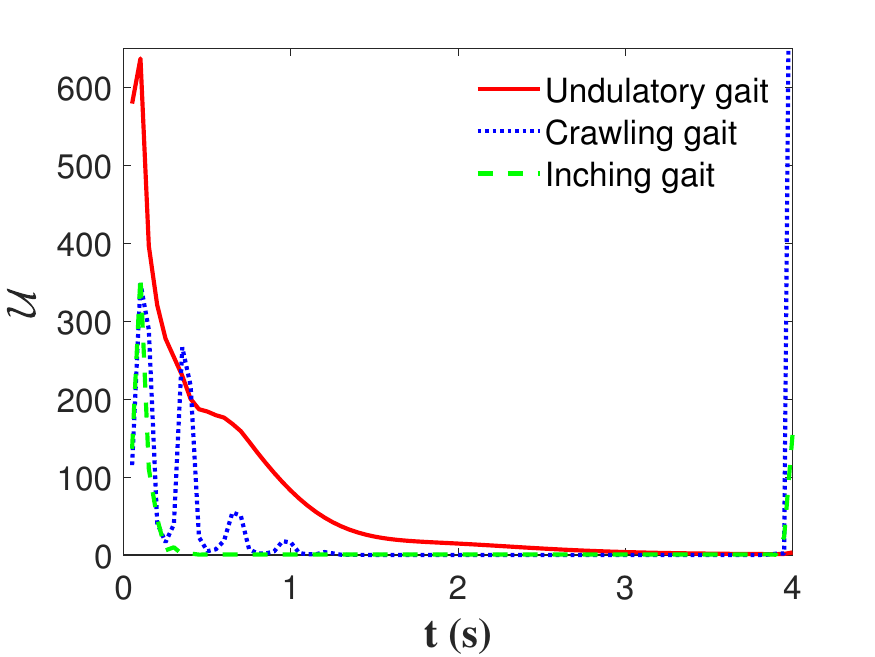}
\includegraphics[width=0.4\textwidth]{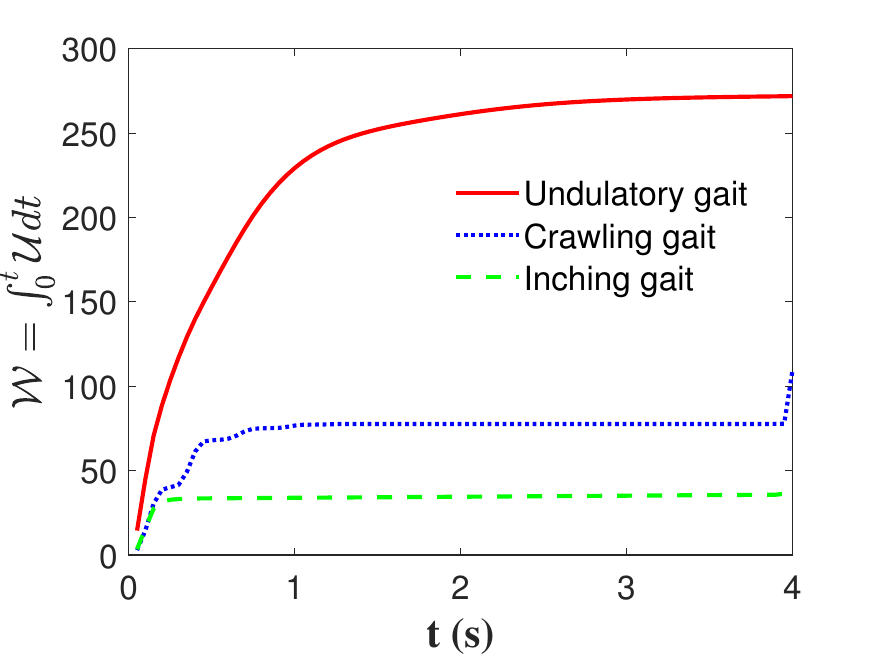}\\
\centerline{(c)\hspace{40ex}(d)}

\caption{Optimal locomotion in the scaled limbless organism: (a) centre of mass ($\x_{cm}$), (b) cost functional ($\mathcal{J}$), (c) internal energy ($\mathcal{U}$), and (d) total energy spent ($\mathcal{W}$).}
\label{fig: 8}
\end{figure}

The optimal centroid evolution, functional minimisation with FBSM iterations and energy expenditure are depicted in Fig. \ref{fig: 8}. The energy expenditure in the scaled organism almost doubled compared to the unscaled case\blue{, but the rest of the trends remained unaltered}. Similarly, Fig. \ref{fig: 9} represents the evolution of internal growth wave distribution and represents the coordinated action of the dorsal-ventral muscles system. \blue{These findings support that inching gait is the most energy efficient and fastest and undulatory gait is the least while crawling falls in between them.}
\begin{figure}[!htb] 
\centering
\includegraphics[width=0.4\textwidth]{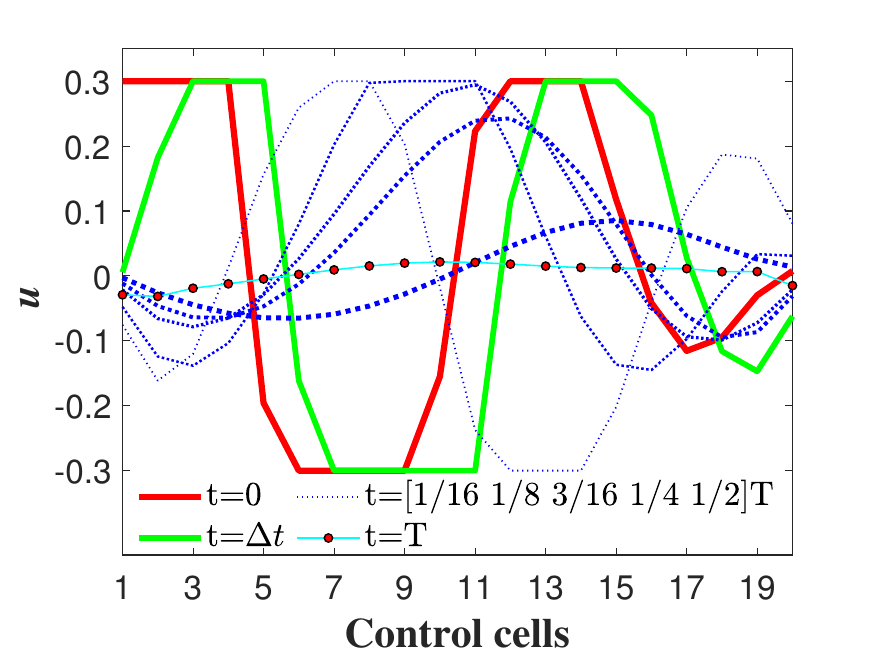} 
\includegraphics[width=0.4\textwidth]{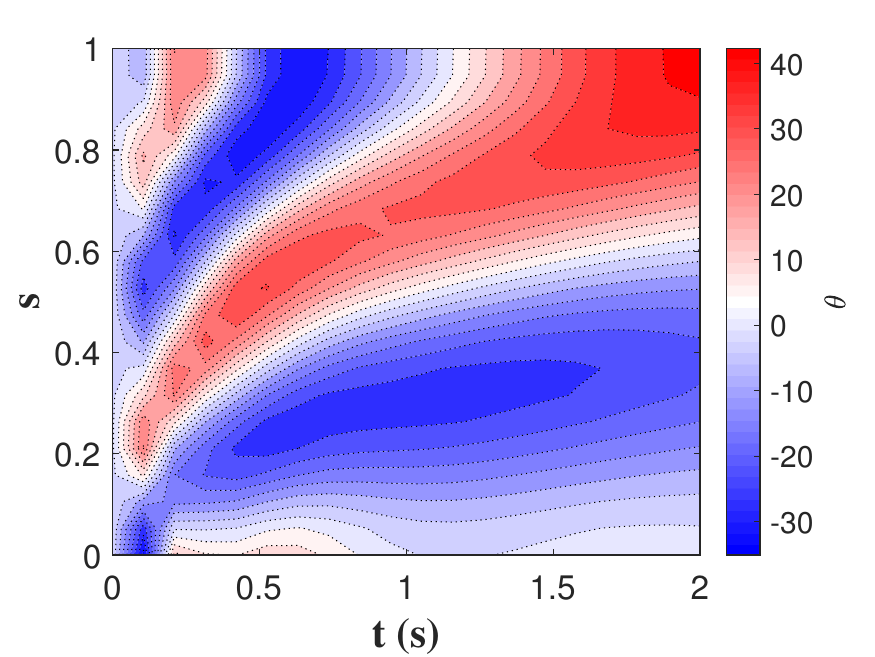}
\includegraphics[width=0.4\textwidth]{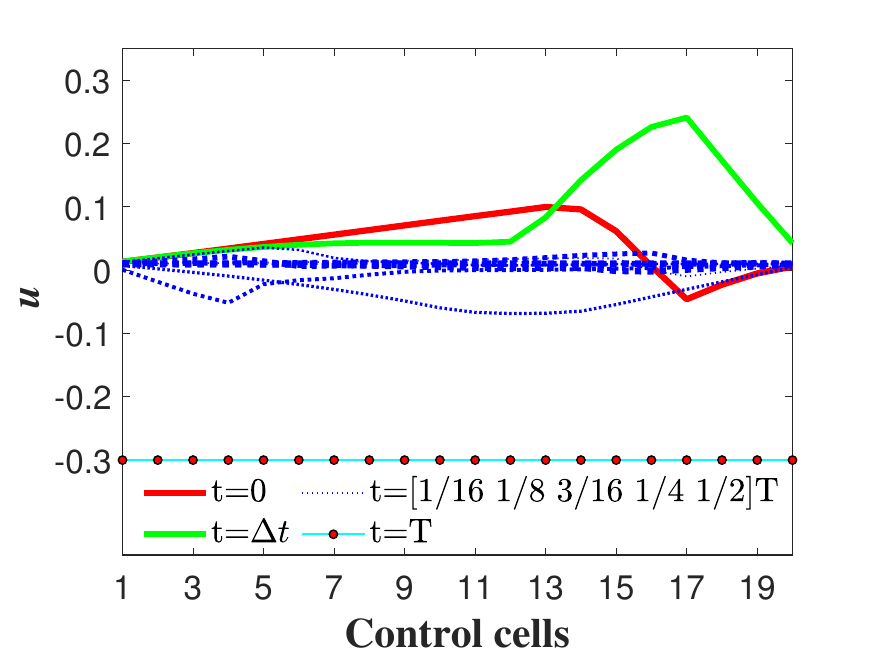}
\includegraphics[width=0.4\textwidth]{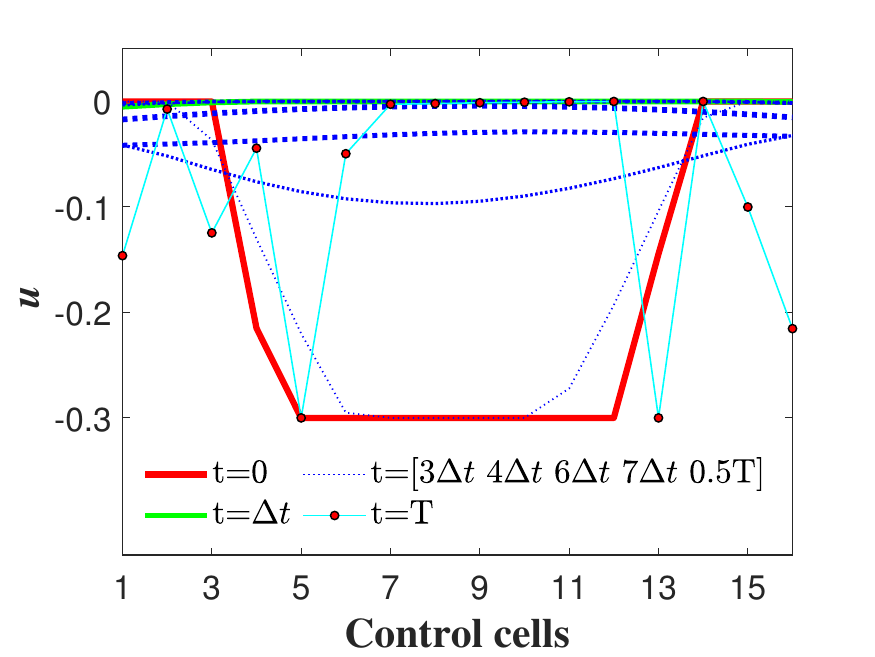}
\caption{Optimal locomotion in the scaled limbless organism: (a) undulatory growth distribution, (b) undulatory body curvature variation, (c) crawling growth distribution, and (d) inching growth distribution.}
\label{fig: 9}
\end{figure}


\section{\blue{Conclusions}}\label{concl}
\label{s:concl}

We presented a unified finite element-based computational framework for simulating the locomotion of limbless 3D soft contractile bodies on \blue{frictional} substrates\blue{, and we have introduced and implemented a strategy for computing optimal growth patterns resorting to optimal control theory}. We believe that the implications of this study will potentially transform the area of computational design and optimal control of limbless soft robots. We have shown with our numerical experiments that the present formulation outcomes are consistent with the locomotion pattern observed in nature. We investigated the role of substrate anisotropy, coordinated action of dorsal-ventral muscles system, and body shape modulation on the net displacement of limbless organism centroid. \blue{These factors } constitute key features of the regulatory mechanism behind the limbless locomotion. Our experiments indicate that for a given growth wave shape, frequency modulation improves the centroid net displacement on anisotropic substrates. We have employed the adjoint-based approach to deduce the sensitivity of objective functional and the efficiency of FBSM has been shown with the GD method integrated with a \blue{Barzikai-Borwein line-search} scheme to accelerate the convergence rate. For tracking type objective functional and considered material model, numerical solution of optimal control problem shows that inching gait is the most energy efficient compared to the undulatory and crawling gait, \blue{when inertial forces and gravitational potential is ignored}.

\blue{The inclusion of inertial and gravitational effects} could potentially alter the energy demands and efficacy of locomotion patterns. Although, the influence of these effects is negligible for the length scale considered here (caterpillar to \blue{\emph{C. elegans}})\blue{, in other applications such as soft robotics or propelled bodies in fluids, may become crucial factors. These effects can be considered in out current framework, which can be formally extended to larger length scales}. In particular, the optimal control problem will be \blue{ODE constrained instead of DAE constrained} optimisation. We discussed the geometric structure of the solution of the optimal control problem and presented the generalised $\tau$-time integration scheme. Although in this work, we have deployed the symplectic Euler time integration scheme, in future studies the presented symplectic structure of the solution could be exploited to design the structure-preserving scheme which could preserve the control Hamiltonian.

%
\section*{Acknowledgement}This work is financially supported by the Spanish Ministry of Science and Innovation, under Severo Ochoa program CEX2018-000797-S, and the research project DynAd2, with reference PID2020-116141GB-I00. The financial support of the local government of Generalitat de Catalunya under grant \blue{2021 SGR 01049 is also acknowledged}.

\bibliographystyle{alpha}
\newcommand{\etalchar}[1]{$^{#1}$}


\appendix
\label{s:appx}

\section{External load vector and tangent matrices}
\label{appendix:a}

\subsection{Shape function and external load vector}
\label{appendix:a1}

\blue{In all our simulations we consider 8 noded linear hexahedral elements, with shape functions defined inthe reference parametric space as ($a=1,2,\dots,m=8$)}
\begin{align}\label{e:g1}
\blue{N^a(\xi, \eta, \zeta)=\frac{1}{8}(1+ \xi )(1\pm \eta )(1\pm \zeta ).}
\end{align}


\blue{In the sequel, we will omit the dependence of shape function $N^a$ on its parametric coordinates.} Unit tangent vector $\tu$ for element $e$ is defined from the positions of nodes $a=1$ and $a=2$. Its expression and its non-zero spatial derivatives are given by,
\begin{align*}
\tu= &\frac{\prescript{e}{}{\x}_2-\prescript{e}{}{\x}_1}{ ||\prescript{e}{}{\x}_2-\prescript{e}{}{\x}_1||}, \\
\frac{\partial \tu} {\partial \prescript{e}{}{\x}_1}= & \frac{1}{||\prescript{e}{}{\x}_2-\prescript{e}{}{\x}_1||} ( -\textbf{I}+\tu \otimes \tu), \\
\frac{\partial \tu} {\partial \prescript{e}{}{\x}_2}= & \frac{1}{||\prescript{e}{}{\x}_2-\prescript{e}{}{\x}_1||} (\textbf{I}-\tu \otimes \tu).
\end{align*}


External load vector $\prescript{e}{}{\g}_{ext}$ at node $a$ of element $e$ is decomposed into surface friction force $\prescript{e}{}{\g}_s$ and bulk viscous/body force $\prescript{e}{}{\g}_v$ (see Eq. \eqref{e:a37}-\eqref{e:gintext})
\begin{align}\label{e:g3}
\prescript{e}{}{\g}_{ext}^a=\prescript{e}{}{\g}_v^a+\prescript{e}{}{\g}_s^a.
\end{align}

Substituting Eq. \eqref{e:a20} and Eq. \eqref{e:a30} into Eq. \eqref{e:gintext}, elemental body force $\prescript{e}{}{\g}_v$ at node $a$ of element $e$ is given by 
\begin{align*}
\prescript{e}{}{\g}_v^a:= \int_{\Omega_e} N^a \zb_o d\Omega_e
= -\int_{\Omega_e} \mu_o N^a \bvf d\Omega_e
= -\int_{\Omega_e} \mu_o N^a N^b d\Omega_e \prescript{e}{}{\bv}^b
= -\prescript{e}{}{\overline{\mathbf{M}}}^{ab} \prescript{e}{}{\bv}^b,
\end{align*}
where $\prescript{e}{}{\overline{\mathbf{M}}}$ is first elemental dissipation matrix defined as 
\[
\prescript{e}{}{\overline{\mathbf{M}}}^{ab}:= \int_{\Omega_e} \mu_o N^a N^b d\Omega_e.
\]

Similarly, substituting Eq. \eqref{e:a15} and Eq. \eqref{e:a30} into Eq. \eqref{e:gintext}, elemental substrate frictional force at node $a$ of element $e$ is given by
\begin{align}\label{e:g5}
\prescript{e}{}{\g}_s^a:= \int_{\Gamma_e} N^a \zt_o d\Gamma_e
= -\int_{\Gamma_e} N^a \mathbf{B} \bvf d\Gamma_e
= -\int_{\Gamma_e} N^a N^b \mathbf{B} d\Gamma_e \prescript{e}{}{\bv}^b.
\end{align}

Using Eq. \eqref{e:g3} and substituting for $\mathbf{B}$ \blue{in Eq. \eqref{e:a16} into Eq. \eqref{e:g5}}, elemental residue vector reduces to
\begin{align}\label{e:g6}
\prescript{e}{}{\g}^a_{ext} = -\left( \prescript{e}{}{\overline{\mathbf{D}}}^{ab}_1 +\prescript{e}{}{\overline{\mathbf{D}}}^{ab}_2 + \prescript{e}{}{\overline{\mathbf{M}}}^{ab} \right) \prescript{e}{}{\bv}^b,
\end{align}
where $\prescript{e}{}{\overline{\mathbf{D}}}^{ab}_1$, and $\prescript{e}{}{\overline{\mathbf{D}}}^{ab}_1$ forms another set of elemental dissipation matrices defined as
\begin{align*}
&\prescript{e}{}{\overline{\mathbf{D}}}^{ab}_1 :=\left(\int_{\Gamma_e} N^a N^b d\Gamma_e \right) \left[\mu_l (\mathbf{I}-\n \otimes \n) + \left(\frac{\mu_f+\mu_b}{2}-\mu_l \right) (\tu \otimes \tu) \right] \\
&\prescript{e}{}{\overline{\mathbf{D}}}^{ab}_2 :=\left(\frac{\mu_f-\mu_b}{2} \right) \left(\int_{\Gamma_e} N^a N^b \tanh{\left(\frac{\bvf \cdot \tu}{\veps} \right)} d\Gamma_e \right) (\tu \otimes \tu).
\end{align*}


\subsection{Linearisation: state tangent matrices}
\label{appendix:a2}

Directional derivative of elemental residue vector $\prescript{e}{}{\g}$ at node $a$ of element $e$ at $\x^d$ and $\bv^d$ along $\eta$ direction is given by
\begin{align*}
\prescript{e}{}{\mathbf{K}}^{ad}\cdot \et:=\frac{d}{d \veps} \prescript{e}{}{\g}^a(\x^d+\veps \et,\bv^d, \bu)\bigr\rvert_{\veps=0}= \mathcal{D}_{\x^d} \prescript{e}{}{\g}^a \cdot \et \\
\prescript{e}{}{\mathbf{G}}^{ad}\cdot \et :=\frac{d}{d \veps} \prescript{e}{}{\g}^a(\x^d,\bv^d+\veps \et, \bu)\bigr\rvert_{\veps=0} = \mathcal{D}_{\bv^d} \prescript{e}{}{\g}^a \cdot \et
\end{align*}

The gradient of the elemental residue vector ($\prescript{e}{}{\g}$) can be decomposed into internal and external tangent contributions and given by
\begin{align*}
\prescript{e}{}{\mathbf{K}}^{ad} &=\mathcal{D}_{\x^d} \prescript{e}{}{\g}^a_{ext} -\mathcal{D}_{\x^d} \prescript{e}{}{\g}^a_{int}, \\
\prescript{e}{}{\mathbf{G}}^{ad} &=\mathcal{D}_{\bv^d} \prescript{e}{}{\g}^a_{ext}.
\end{align*}

\subsubsection{External load tangent matrix}

Tangent contribution of external load vector $\prescript{e}{}{\g}^a_{ext}$ due to perturbation in spatial nodal position $\x^d$ and spatial nodal velocity $\bv^d$ can be evaluated as
\begin{align*}
-\mathcal{D}_{\x^d} \prescript{e}{}{\g}^a_{ext}= & \int_{\Gamma_e} N^a \mathbf{H}^d d\Gamma_e, \\
-\mathcal{D}_{\bv^d} \prescript{e}{}{\g}^a_{ext} =& \prescript{e}{}{\overline{\mathbf{D}}}^{ad}_1 +\prescript{e}{}{\overline{\mathbf{D}}}^{ad}_2 + \prescript{e}{}{\overline{\mathbf{M}}}^{ad} + \int_{\Gamma_e} N^a \mathbf{W}^d d\Gamma_e,
\end{align*}

where integrands are defined as
\begin{align*}
\mathbf{H}^d =& \left[ \left( \frac{\mu_f+\mu_b}{2}-\mu_l \right) + \left(\frac{\mu_f -\mu_b} {2} \right) \text{tanh}{\left(\frac{\bvf \cdot \tu}{\veps} \right)} \right] ((\tu \cdot \bvf)\mathbf{I}+ \tu \otimes \bvf) \frac{\partial \tu} {\partial \prescript{e}{}{\x}^{d}}, \\ 
+& \left(\frac{\mu_f -\mu_b} {2 \veps } \right) \text{sech}^2{\left(\frac{\bvf \cdot \tu}{\veps} \right)} (\tu \cdot \bvf) (\tu \otimes \bvf) \frac{\partial \tu} {\partial \prescript{e}{}{\x}^{d}}, \\
\mathbf{W}^d =& \left(\frac{\mu_f -\mu_b} {2 \veps } \right) \text{sech}^2{\left(\frac{\bvf \cdot \tu}{\veps} \right)} (\tu \cdot \bvf) (\tu \otimes \tu)N^d.
\end{align*}


\subsubsection{Internal load tangent matrix}

The elemental tangent matrix of the internal load vector at node $a$ w.r.t to perturbation in the nodal spatial position of node $d$ of element $e$ is given by (no summation on growth label $g$)
\begin{align}\label{e:h5}
\mathcal{D}_{\x^{d}_{k}} (\prescript{e}{}{\g_{int})}^a_{i} = \int_{\Omega_e} J_g (\zF_g^{-T})_{jm} \frac{\partial N^a}{\partial \zx_m} (\mathbb{A}_e)_{ijkl} (\zF_g^{-T})_{ls} \frac{\partial N^d}{\partial \zx_s} d\Omega_e.
\end{align}

For Neo-Hookean solid, \blue{by substituting referential tangent modulus in Eq. \eqref{e:a11} into Eq. \eqref{e:h5} we obtain after simplification }
\begin{align}\label{e:h6}
\mathcal{D}_{\x^{d}_{k}} (\prescript{e}{}{\g_{int})}^a_{i}=& \int_{\Omega_e} J_g \mu \left(\zF_g^{-T} \nabla_{\zx} N^a \right)_j \left(\zF_g^{-T} \nabla_{\zx} N^b \right)_j d\Omega_e \delta_{ik} + \int_{\Omega_e} J_g \lambda \left(\zF^{-T} \nabla_{\zx} N^a \right)_i \left(\zF^{-T} \nabla_{\zx} N^b \right)_k d\Omega_e \\ \nonumber
& - \int_{\Omega_e} J_g (\lambda \ln{J_e}-\mu) \left(\zF^{-T} \nabla_{\zx} N^b \right)_i \left(\zF^{-T} \nabla_{\zx} N^a \right)_k d\Omega_e.
\end{align}

Introducing scalar and tensor products between vector entities, the elemental tangent matrix due to internal load can be written in the following simplified form
\begin{align}\label{e:h7}
\mathcal{D}_{\x^d} \prescript{e}{}{\g}^a_{int} = & \int_{\Omega_e} J_g \mu \left(\zF_g^{-T} \nabla_{\zx} N^a \right) \cdot \left(\zF_g^{-T} \nabla_{\zx} N^b \right) d\Omega_e \mathbf{I} + \int_{\Omega_e} J_g \lambda \left(\zF^{-T} \nabla_{\zx} N^a \right) \otimes \left(\zF^{-T} \nabla_{\zx} N^b \right) d\Omega_e \\ \nonumber
& - \int_{\Omega_e} J_g (\lambda \ln{J_e}-\mu) \left(\zF^{-T} \nabla_{\zx} N^b \right) \otimes \left(\zF^{-T} \nabla_{\zx} N^a \right) d\Omega_e.
\end{align}


\subsection{Fiber orientation and control tangent matrix}
\label{appendix:a4}

Let us consider a muscles fibre oriented along direction $\zi$ that makes angle $\omega_X$, $\omega_Y$, and $\omega_Z$ respectively with \blue{respect to the the positive directions of $X$, $Y$, and $Z$ axes} of the reference frame. Then, fibre direction expressed in terms of direction cosines and basis vector $\E$ reads
\begin{align}\label{e:j1}
\zi=\cos\omega_X \E_X +\cos\omega_Y \E_Y+\cos\omega_Z \E_Z,
\end{align}


Growth deformation gradient $\zF_g \in \mathit{Sym}$ can be \blue{written} in terms of structural tensor $\mathbf{A}=\zi \otimes \zi$ as
\begin{align*}
\zF_g =\textbf{I}+u \mathbf{A}=
\begin{bmatrix}
1+u \cos^2\omega_X & u \cos\omega_X\cos\omega_Y & u \cos\omega_X\cos\omega_Z \\
u \cos\omega_X\cos\omega_Y & 1+u \cos^2\omega_Y & u \cos\omega_Y\cos\omega_Z \\
u \cos\omega_X\cos\omega_Z & u \cos\omega_Y\cos\omega_Z & 1+u \cos^2\omega_Z \\
\end{bmatrix}.
\end{align*}

Cofactor matrix of growth deformation gradient \blue{$\text{Cof}\,\zF_g:=J_g \zF_g^{-T}$ can be then} expended as
\begin{align*}
\text{Cof}\,\zF_g = \begin{bmatrix}
1+u (\cos^2\omega_Y +\cos^2\omega_Z) & - u \cos\omega_X\cos\omega_Y & -u \cos\omega_X\cos\omega_Z \\
-u \cos\omega_X\cos\omega_Y & 1+u (\cos^2\omega_X+\cos^2\omega_Z) & -u \cos\omega_Y\cos\omega_Z \\
-u \cos\omega_X\cos\omega_Z & -u \cos\omega_Y\cos\omega_Z & 1+u (\cos^2\omega_X+\cos^2\omega_Y) \\
\end{bmatrix},
\end{align*}
and using the orthonormal property of direction cosines, the determinant of $\zF_g$ reduces just to the growth-induced stretching, i.e., $J_g=1+u$. Directional derivative of $\zF_g$ and $\text{Cof}\,\zF_g$ w.r.t $u$ along $\et$ direction is given by
\begin{align}\label{e:j4}
\mathcal{D}_u \zF_g \cdot \et =& \begin{bmatrix}
\cos^2\omega_X & \cos\omega_X\cos\omega_Y & \cos\omega_X\cos\omega_Z \\
\cos\omega_X\cos\omega_Y & \cos^2\omega_Y & \cos\omega_Y\cos\omega_Z \\
\cos\omega_X\cos\omega_Z & \cos\omega_Y\cos\omega_Z & \cos^2\omega_Z \\
\end{bmatrix} \cdot \et \\ \nonumber
\mathcal{D}_u \text{Cof}\,\zF_g \cdot \et =& \begin{bmatrix}
\cos^2\omega_Y +\cos^2\omega_Z & - \cos\omega_X\cos\omega_Y & - \cos\omega_X\cos\omega_Z \\
- \cos\omega_X\cos\omega_Y & \cos^2\omega_X+\cos^2\omega_Z & - \cos\omega_Y\cos\omega_Z \\
- \cos\omega_X\cos\omega_Z & - \cos\omega_Y\cos\omega_Z & \cos^2\omega_X+\cos^2\omega_Y \\
\end{bmatrix} \cdot \et \\ \nonumber
=& (\textbf{I}-\mathcal{D}_u \zF_g) \cdot \et =:\textbf{J} \cdot \et \\ \nonumber
\mathcal{D}_u \zF_g^\mathsf{-1} \cdot \et =& \frac{1}{J_g} (\textbf{J}-\zF_g^\mathsf{-1}) \cdot \et \nonumber.
\end{align}

It is important to note that the nodal position $\x$ and growth $\bu$ act as independent variables in the optimal control formulation. For prescribed spatial position ($\x$), perturbation of intermediate configuration ($\bu \rightarrow \bu+ \veps \et$) leads to the perturbation of elastic state variables ($\zF_e$). Then, directional derivative of $\ln{J_e}$, and $\zF_e$ w.r.t $u$ along $\et$ direction is given by
\begin{align}\label{e:j5}
\mathcal{D}_u \ln{J_e} \cdot \et =& \frac{1}{J_e} \mathcal{D}_u J_e \cdot \et = -\frac{1}{J_g} \et \\ \nonumber
\mathcal{D}_u \zF_e \cdot \et= & \zF \mathcal{D}_u \zF_g^\mathsf{-1} \cdot \et = \frac{1}{J_g} \zF (\textbf{J}-\zF_g^\mathsf{-1}) \cdot \et \\ \nonumber
\mathcal{D}_u \zF_e^\mathsf{-T} \cdot \et= & \zF^\mathsf{-T} \mathcal{D}_u \zF_g^\mathsf{T} \cdot \et = -\zF^\mathsf{-T} (\textbf{J}-\textbf{I}) \cdot \et \nonumber
\end{align}

Directional derivative of elemental internal load vector at node $a$ w.r.t elemental growth $\prescript{e}{}{u}$ along $\et$ direction is given by (See Eq. \eqref{e:a36})
\begin{align}\label{e:j6}
\prescript{e}{}{\B}^{a} \cdot \et=\mathcal{D}_u \prescript{e}{}{\g_{int}^{a}} \cdot \et=\int_{\Omega_{e}} (J_g \mathcal{D}_u \zP_e \cdot \et) \zF_g^{-T} \nabla_{\zx} N^a d\Omega_{e}+ \int_{\Omega_{e}} \zP_e (\mathcal{D}_u \text{Cof}\,\zF_g \cdot \et) \nabla_{\zx} N^a d\Omega_{e}.
\end{align}

Using Eq. \eqref{e:j4} and Eq. \eqref{e:j5}, the directional derivative of elastic first Piola-Kirchhoff tensor for the Neo-Hookean solid is given by (See Eq. \eqref{e:a10})
\begin{align}\label{e:j7}
\mathcal{D}_u \zP_e \cdot \et =& \mu (\mathcal{D}_u\zF_e \cdot \et)+(\lambda \ln{J_e}-\mu) (\mathcal{D}_u \zF_e^\mathsf{-T} \cdot \et) + \lambda \zF_e^\mathsf{-T} (\mathcal{D}_u \ln{J_e} \cdot \et) \\ \nonumber
\Rightarrow J_g \mathcal{D}_u \zP_e \cdot \et =& \left(\mu \zF (\textbf{J}-\zF_g^\mathsf{-1}) - (\lambda \ln{J_e}-\mu) \zF^\mathsf{-T} (\textbf{J}-\zF_g^\mathsf{-1}) - \lambda \zF_e^\mathsf{-T} \right) \cdot \et \nonumber
\end{align}

Substituting Eq. \eqref{e:j4} and Eq. \eqref{e:j7} in Eq. \eqref{e:j6}, elemental control tangent matrix is given by
\begin{align}\label{e:j8}
\prescript{e}{}{\B}^{a}= \int_{\Omega_{e}} \left(\mu \zF (\textbf{J}-\zF_g^\mathsf{-1}) - (\lambda \ln{J_e}-\mu) \zF^\mathsf{-T} (\textbf{J}-\zF_g^\mathsf{-1}) - \lambda \zF_e^\mathsf{-T} \right) J_g \zF_g^{-T} \nabla_{\zx} N^a d\Omega_{e} + \int_{\Omega_{e}} \zP_e \textbf{J} \nabla_{\zx} N^a d\Omega_{e}
\end{align}

For unidirectional muscles actuation (along $X$ direction: $\hat{\zi}=\E_X$ with $\omega_Y=\omega_Z=\frac{\pi}{2}$) leads to the further simplification of control tangent matrix integrands. For instance,
\begin{align*}
\zF_g &=\text{diag}(1+u,1,1), \\
\textbf{J}&=\mathcal{D}_u \text{Cof}\,\zF_g =\text{diag}(0,1,1), \\
J_g \mathcal{D}_u \zP_e &=\frac{\mu}{J_g} \zF(\textbf{J}-\textbf{I})-J_g (\lambda \ln{J_e}-\mu)\zF^{-T} (\textbf{J}-\textbf{I}) -\lambda \zF_e^{-T}.
\end{align*}

Suppose \blue{elements $e^+$ and $e^-$ $\in \{1,2,\ldots,E \} $ share} equal and opposite growth with $u^+=u$ and $u^-=- u$, respectively. Using the above relations, dipole contributions to the elemental control tangent matrix read ($u^+$ is treated as an independent variable)
\begin{align*}
\prescript{e+}{}{\B}^{a}=\frac{\partial \prescript{e^+}{}{\g_{int}}^{a}}{\partial u^+} &= \int_{\Omega_{e^+}} (J_g \mathcal{D}_u \zP_e ) \zF_g^{-T} \nabla_{\zx} N^a d\Omega_{e^+}+ \int_{\Omega_{e^+}} \zP_e \textbf{J}_x \nabla_{\zx} N^a d\Omega_{e^+}, \\ 
\prescript{e-}{}{\B}^{a}=\frac{\partial \prescript{e^-}{}{\g_{int}}^{a}}{\partial u^+} &= -\int_{\Omega_{e^-}} (J_g \mathcal{D}_u \zP_e) \zF_g^{-T} \nabla_{\zx} N^a d\Omega_{e^-} - \int_{\Omega_{e^-}} \zP_e \textbf{J}_x \nabla_{\zx} N^a d\Omega_{e^-}.
\end{align*}


\section{Barzikai-Borwein \blue{line-search} method}\label{appendix:b}

Let us consider the reduced minimisation problem Eq. \eqref{e:b11}, we update the decision variable $\bu^{k+1}$ by moving $\theta^k$ along the search direction $\dd^k:=-\nabla_{\bu} \hat{\mathcal{J}} (\bu^k)$
\begin{align*}
\bu^{k+1}=\bu^k+\theta^k \dd^k.
\end{align*}

With the analogy of the Newton-Raphson method, one can consider the following equivalent form
\begin{align*}
\bu^{k+1}=\bu^k+\theta^k \textbf{I} \dd^k=\bu^k- \textbf{H}_k^{-1} (-\dd^k),
\end{align*}
where $\textbf{H}_k^{-1}=\theta^k \textbf{I}$ or $\textbf{H}_k=\theta_k^{-1} \textbf{I}$. Following the Quasi-Newton approach, Hessian at $\bu^k$ can be approximated with a two-point secant approximation
\begin{align}\label{e:F3}
\textbf{H}_k \Delta \bu \approx - \Delta \dd,
\end{align}
where $\Delta \bu:=\bu^k-\bu^{k-1}$, and $\Delta \dd:=\dd^k-\dd^{k-1}=\nabla_{\bu} \hat{\mathcal{J}} (\bu^{k-1})-\nabla_{\bu} \hat{\mathcal{J}} (\bu^k)$. With Eq. \eqref{e:F3}, the optimal step size can be obtained in the following two ways \cite{barzilai88}:

\textbf{Longer step size ($\theta_L^k$)}: With $\textbf{H}_k=(\theta^k)^{-1} \textbf{I}$, then $\theta_L^k$ can be obtained in a least square sense such that $\textbf{H}_k \Delta \bu + \Delta \dd \approx \0$
\begin{align*}
\theta_L^k= \argmin_{\theta \in \mathbb{R}^\mathsf{+}} & ||\theta^{-1} \Delta \bu + \Delta \dd||_{L^2}=- \frac{\Delta \bu ^\mathsf{T} \Delta \bu} {\Delta \bu ^\mathsf{T} \Delta \dd}.
\end{align*}

\textbf{Shorter step size ($\theta_s^k$)}: From symmetry, $\textbf{H}_k^{-1}=\theta^k \textbf{I}$, then $\theta_s^k$ can be obtained in a least square sense such that $\Delta \bu + \textbf{H}_k^{-1} \Delta \dd \approx \0$
\begin{align*}
\theta_s^k= \argmin_{\theta \in \mathbb{R}^\mathsf{+}} & ||\Delta \bu + \theta \Delta \dd||_{L^2}=- \frac{\Delta \bu ^\mathsf{T} \Delta \dd} {\Delta \dd ^\mathsf{T} \Delta \dd }.
\end{align*}


\subsection{Remedy for negative step size}

Barzikai-Borwein algorithm with short step size $\theta_s^k$ performs numerically better than earlier long step size $\theta_L^k$. Unfortunately, for non-convex functions, the Barzikai-Borwein algorithm may generate a negative step length. In \cite{dai15}, a lower bound on step length $\theta_m^{k}$ is proposed, which is computed as a geometric mean of the shortest and longest step size:
\begin{align}\label{e:F6}
\theta_m^{k} := \left(\theta_s^k \theta_L^k \right)^{1/2}
=\frac{||\Delta \bu ||} {||\Delta \dd ||}
\end{align}

\blue{Consequently, the step parameter $\theta^k$ is updated according to}
\begin{equation}\label{e:X4}
\theta^k= \begin{cases}
\theta_s^k, & \text{if } \theta_s^k>0, \\
\theta_m^k , & \text{if } \theta_s^k<0.
\end{cases}
\end{equation}

\blue{Algorithm \ref{a:2} describes the update process of the \blue{line-search} employed in our examples}.

\begin{algorithm}
\caption{Stabilized Barzikai-Borwein algorithm}
\begin{algorithmic}
\Require $\theta^{max}=1, \bu^{k-1},\bu^k, \dd^{k-1},\dd^k$

\State $\theta_{th}^k \leftarrow \frac{\theta_{max}}{|| \dd^k||}$ \Comment{Threshold step length}
\State $\Delta\dd \leftarrow \dd^{k}-\dd^{k-1}$
\State $\Delta\bu \leftarrow \bu^{k}-\bu^{k-1}$
\State $||\Delta\dd|| \leftarrow \sqrt{\Delta\dd^\mathsf{T} \Delta\dd}$
\State $\blue{\theta_s} \leftarrow -\frac{\Delta\bu^\mathsf{T} \Delta\dd }{||\Delta\dd||^2}$ \Comment{Short step length}
\If{$\theta_s^k > 0$}
\State $ \theta^k \leftarrow \min(\theta_s^k,\theta_{th}^k) $ \Comment{Stabilization step}
\Else
\State $\theta_m^k \leftarrow \frac{||\Delta\bu||}{||\Delta\dd||}$ \Comment{Correction for negative step length}
\State $ \theta^k \leftarrow \min(\theta_m^k,\theta_{th}^k) $ \Comment{Stabilization step}
\EndIf
\Ensure $\theta^k $ \Comment {\blue{Line-search} parameter}
\end{algorithmic}
\label{a:2}
\end{algorithm}


\subsection{\blue{Stabilised} Barzikai-Borwein algorithm}

It is important to note that the Barzikai-Borwein method converges R-super-linearly for strictly convex quadratics in 2 or 3 dimensions \cite{barzilai88}. However, for the general n-dimensional objective function, the Barzikai-Borwein method converges globally with R-linear rate \cite{burdakov19, yahaya21}. It has been proven that the Barzikai-Borwein method significantly improves the convergence rate of the GD algorithm. However, a major drawback is that sometimes the Barzikai-Borwein method generates too long step length and the GD algorithm may not converge even for strongly convex functions. For such a scenario, a stabilisation technique is introduced in \cite{burdakov19} to improve the efficiency of the Barzikai-Borwein method and known as stabilized Barzikai-Borwein algorithm. The idea is to restrict the maximum allowed step length increment per iteration i.e. for some $\theta_{max}>0$, whenever $||\theta^k \dd^k|| >\theta_{max} $, impose
\begin{align}\label{e:F7}
||\bu^{k+1}-\bu^k ||= \theta_{max}.
\end{align}

One can conclude, threshold value of $\theta_k$ can be obtained as $\theta_{th}^k=\frac{\theta_{max}}{|| \dd^k|| }$. Hence, the \blue{line-search} parameter can be updated as
\begin{equation}\label{e:X4}
\theta^k= \begin{cases}
\min(\theta_s^k,\theta_{th}^k), & \text{if } \theta_s^k>0, \\
\min(\theta_m^k,\theta_{th}^k) , & \text{if } \theta_s^k<0.
\end{cases}
\end{equation}

\end{document}